\documentclass[a4paper,12pt]{article}
\usepackage{latexsym}
\usepackage[latin1]{inputenc}
  \usepackage[T1]{fontenc}
\usepackage{amsmath,amssymb,amsfonts}
\usepackage{amsthm}
\everymath{\displaystyle}
\input xy
\xyoption{all}
\usepackage[all]{xy}
\input xypic
\xyoption{v2}
\newtheorem{thm}{Theorem}[section]
\newtheorem{prop}[thm]{Proposition}
\newtheorem{lemma}[thm]{Lemma}
\newtheorem{cor}[thm]{Corollary}
\newtheorem{defn}[thm]{Definition}
 
 \newtheorem{ex}[thm]{Example}
 \theoremstyle{remark}

\def\gal{\operatorname{Gal}}
\def\ker{\operatorname{Ker}}
\def\coker{\operatorname{Coker}}
\def\hom{\operatorname{Hom}}
\def\tor{\operatorname{tor}}
\def\fr{\operatorname{fr}}
\def\im{\operatorname{im}}

\def\div{\operatorname{Div}}
\def\rank{\operatorname{rank}}
\def\corank{\operatorname{corank}}

\newcommand{\F}{\mathbb{F}}

\newcommand{\Q}{\mathbb{Q}}
\newcommand{\Z}{\mathbb{Z}}

\newcommand{\Zp}{\Z_p}
\newcommand{\QpZp}{\Q_p/\Z_p}

{\setlength{\parindent}{0mm}
{\setlength{\parskip}{5mm}

\begin{document}
\author{A. MOVAHHEDI and T. NGUYEN QUANG DO}
\title{On universal norms and the first layers of $\Z_p$-extensions of a number field}
\maketitle

\abstract 

For an odd prime  $p$  and a number field  $F$  containing a primitive $p$-th root of unity, we describe 
the Kummer radical  ${\cal A}_F$  of the first layers of all the  $\Z_p$-extensions of  $F$  
in terms of universal norms of  $p$-units along the cyclotomic tower of  $F$. 
We also study "twisted" radicals related to  ${\cal A}_F$.

\section*{Introduction}
 Let $p$ be an odd prime number and  $F$  be a number field containing a primitive $p$-th root of unity  $\zeta_p$. 
 We denote by  ${\cal A}_F$  the Kummer radical of the first layers of all the $\Z_p$-extensions of $F$. Precisely  ${\cal A}_F$  is the subgroup of  $F^\bullet/ F^{\bullet p}$ consisting of classes  $a$ mod  $F^{\bullet p}$  such that the Kummer extension  $F(\sqrt[p]{a})$  is contained in a  $\Z_p$-extension of $F$. 
 The determination of the group  ${\cal A}_F$  is an old problem which dates back to the beginnings of Iwasawa theory and has since been tackled by many authors. 
 Here we present what we hope to be a satisfactory Iwasawa theoretical solution. 
 In order to give meaning to this assertion, it is appropriate to recall that most of the (numerous) results obtained so far bring  the study of  ${\cal A}_F$  back to that of its orthogonal complement for the Kummer pairing, 
 namely  $\tor_{\Z_p} {\cal X}_F/p$,  where 
${\cal X}_F$ denotes the Galois group of the maximal abelian pro-$p$-extension of  $F$  which is unramified outside  $p$-adic primes. Without pretending to be exhaustive, let us cite the following articles:  \cite {CK76},  which uses the idelic description of  ${\cal X}_F$  to compute  $\tor_{\Z_p} {\cal X}$  for certain quadratic fields; 
\cite  {Gra85},  which constructs from Artin symbols a "logarithm" map on  ${\cal X}_F$  whose kernel is precisely the  $\Z_p$-torsion $\tor_{\Z_p} {\cal X}_F$; \cite  {He88, Th93}, which approach ${\cal A}_F$ by a "d\'evissage" of $\tor_{\Z_p} {\cal X}_F$  in a local-global perspective $\cdots$. From a cohomological point of view, and under Leopoldt's conjecture,  $\tor_{\Z_p} {\cal X}_F$  is a "twisted dual" (see for example   \cite {Ng86} and the references therein)  of the $p$-primary part of the tame kernel of  $K_2F$ and a question raised by Coates   \cite {Co73}  was whether  ${\cal A}_F$  coincides with the Tate kernel of  $F$, i.e. the subgroup  ${\cal T}_F$  of  $F^\bullet/ F^{\bullet p}$ consisting of classes  $a$ mod  $F^{\bullet p}$  such that the symbol  $\{\zeta_p, a\}=0$  in $K_2F$. In this direction and still without claim of exhaustiveness, let us cite the following articles:  
\cite {KC78}, which establishes a "wrong duality"  between the elements of order  $p$ in  ${\cal X}_F$  and the quotient mod  $p$  of the tame kernel  $R_2F$;  \cite {Gre78},  which performs Iwasawa descent on the "twisted duals"  of the free part of the Galois group ${\cal X}_{F_\infty}$, where $F_\infty$ is the cyclotomic $\Z_p$-extension of $F$. 

All these approaches, by class field theory or by $K$-theory, produce effective methods allowing to compute  ${\cal A}_F$  from  arithmetical parameters attached to  $F$,  such as the class group or the group of units of the field. But these descriptions of  ${\cal A}_F$  can not be considered as theoretically complete when they lead to other arithmetical objects such as  $\tor_{\Z_p} {\cal X}_F$   or  ${\cal T}_F$  which are not necessarily better known than  ${\cal A}_F$  itself. Thus 
$\tor_{\Z_p} {\cal X}_F$  is, by Iwasawa theory, linked to the  $p$-adic $L$-functions and its "local-global d\'evissage"  \cite {He88, Ng86, Th93} leads to an Iwasawa module of descent closely related to the class group. In the same way, the "local-global d\'evissage" of the Tate kernel  ${\cal T}_F$  eventually leads to the wild kernel, which is in turn isomorphic to a twisted version of the preceding Iwasawa module of descent. 
In addition, the intervention of an Iwasawa descent suggests that a satisfactory theoretical description of the Kummer radical   ${\cal A}_F$  
should include in some way an asymptotical ingredient.  
This is confirmed by Greenberg's answer to the question of Coates  \cite [page 1242]{Gre78} :  
though on the ground level  ${\cal A}_F$  and  ${\cal T}_F$ are not the same in general, they coincide when going sufficiently   up the cyclotomic tower. In short, a satisfactory (in our sense) description of  ${\cal A}_F$  should take into account  the following two remarks: \\
- the parameters which intervene should be arithmetically or at least effectively accessible \\
- the answer should be allowed to incorporate asymptotical ingredients, i.e. coming from  high enough  $F_n$  (for an explicit or computable  $n$.)

In this paper, we introduce as parameters some accessible norm subgroups of the  pro-$p$-completion  ${\bar U}_F$  of the group of  $p$-units of  $F$, more precisely, the subgroup  ${\tilde U}_F$ of  (global) universal norms in the cyclotomic  $\Z_p$-extension  
$F_\infty/F$, as well as the subgroup  ${\hat U}_F$ of those which are locally universal norms in the local cyclotomic  $\Z_p$-extensions  $F_{v, \infty}/F_v$  at all $p$-adic primes $v$ (hence at all finite primes):   ${\tilde U}_F \subset {\hat U}_F \subset {\bar U}_F$. 
It is known (\cite {FN91, MS03, Se11}) that every element of   ${\tilde U}_F$  starts a  $\Z_p$-extension, i.e.  
${\tilde U}_F  F^{\bullet p}/ F^{\bullet p}  \subset  {\cal A}_F$. 
Our goal is to compare  ${\cal A}_F$  with various radicals derived from   ${\tilde U}_F$ and  ${\hat U}_F$. 

As a first step, we  bound the Kummer radical   ${\cal A}_F$  "from below" by the radicals  ${\tilde U}_F  F^{\bullet p}/ F^{\bullet p}$  and  ${\cal A}_F \cap ({\hat U}_F/p)$  and describe the deviations in terms of some asymptotic capitulation kernels (Corollary 2.5 and Proposition 2.7). Then, we give an "upper bound" for  ${\cal A}_F$  in terms of the fixed points  $({\hat U}_{F_n}/p)^{G_n}$ for  $n >>0$ (but accessible) with  $G_n =\gal(F_n/F)$,  and describe the deviation between them (Theorem 3.1).  
Finally, we introduce a radical  ${\cal B}_F$  defined by conditions of  $\Z_p$-embeddability locally everywhere, which contains the previous radicals  ${\cal A}_F $, ${\cal T}_F$  and  ${\hat U}_F/p$,  and we determine the three respective quotients in Iwasawa theoretical terms. 
Since  the modules  ${\hat U}_{F_n}$  (resp. the capitulation kernels)  are immediately (resp.  asymptotically and effectively) accessible,  our goal has been reached.

Although our approach has a theoretical orientation, it lends itself to effective or algorithmic calculations, as will be shown in the examples of section  $4$, where we shall study the three radicals  ${\cal A}_F $, ${\cal T}_F$  and  ${\hat U}_F/p$  for  $p=3$ and biquadratic fields of the form   $F=\Q(\mu_3, \sqrt{d})$.

\bigskip 
We gather here the notations used in the text.  
Note that not all of them agree with the usual notations in Iwasawa theory as appearing for instance in Iwasawa's paper  \cite {Iw73}. 
 
\vspace{2mm} 
\begin{tabular}{rl}
$F$						& our base number field; \\
$F^\bullet$				& multiplicative group of non-zero elements of  $F$; \\
$r_1$, $r_2$				& number of real (resp. non-conjugate complex)  embeddings of $F$; \\
$F_{\infty} =\cup F_n$		& cyclotomic $\Z_p$-extension of $F$, with finite layers $F_n$; \\
$s$, $s_n$					& number of $p$-adic primes in $F$ (resp. in $F_n$); \\
$\mu_{p^n}$, $\mu_{p^\infty}$  & group of $p^n$-th (resp. all $p$-primary) roots of unity; \\
$\mu_(F)= \mu_{p^\infty} \cap F$  				& group of all  $p$-primary roots of unity contained in $F$; \\
$U_F$, $U_n$				& group of $(p)$-units in $F$ (resp. in $F_n$); \\
${\bar U}_F =\varprojlim (U_F/p^m)$, ${\bar U}_n$ 	&  pro-$p$-completion of  $U_F$ (resp. of $U_n$); \\
$F_v$    & completion of  $F$ at a prime  $v$   in  $F$; \\
${\hat U}_F$,   ${\hat U}_n$  &  subgroup of  ${\bar U}_F$ (resp.  ${\bar U}_n$)  consisting of elements which are universal \\  
& norms from the cyclotomic  $\Z_p$-extensions  $F_{\infty,v}$  for all finite primes  $v$; \\

${\tilde U}_F$,  ${\tilde U}_n$  &  subgroup of  ${\bar U}_F$ (resp.  ${\bar U}_n$)  consisting of universal norms from  $F_{\infty}$; \\

$V_n$  &  factor group ${\bar U}_n/{\tilde U}_n$ \\

${\bar U}_{\infty}= \varprojlim {\bar U}_n$ 	& inverse limit with respect to norm maps; \\ 
$\Lambda=\Z_p[[\Gamma]]\simeq \Z_p[[T]]$ &  Iwasawa algebra, the isomorphism being obtained by mapping $\gamma$ to $1+T$; \\
$S$                  & set of all places of $F$ over $p$; \\ 
$A_F$, $A_n$               & $p$-primary part of the $(p)$-class group of $F$ (resp.  of $F_n$); \\
$A_{\infty}=\varinjlim A_n$ &  $p$-primary part of the $(p)$-class group
 of $F_{\infty}$; \\
$L_n$, $L_{\infty}$   & maximal unramified abelian pro-$p$-extension of $F_n$ \\
               & (resp. of $F_{\infty}$) in which every prime over $p$ splits completely; \\

${\hat F}$   & maximal abelian pro-$p$-extension of $F$ unramified outside the $p$-adic primes;\\
$F^{BP}$   & field of Bertrandias-Payan over $F$, \emph{i.e.},  the compositum of all $p$-extensions  \\
		  & of $F$ which are infinitely embeddable in cyclic $p$-extensions; \\      
$\bar{F}_v^{\bullet}$  &  pro-$p$-completion of  $F_v^{\bullet}$; \\
$\tilde{F}_v^{\bullet}$  & group of universal norms in the cyclotomic $\Z_p$-extension $F_{\infty,v}/F_v$; \\
${\tilde {\cal F}} = {\small \prod} \tilde{F}_v^{\bullet}$ \\

\end{tabular}

\begin{tabular}{rl}
\label{not2}

$\Gamma$,  $\Gamma_n$  & Galois group  $\gal (F_{\infty}/F)$  (resp.  $\gal (F_{\infty}/F_n)$); \\ 
$\gamma$ &  a fixed topological generator of $\Gamma$; \\
$G_n \simeq \Z/p^n\Z$   & Galois group  $\gal (F_n/F)$; \\ 
$BP_F$ & Galois group $\gal (F^{BP}/F)$; \\
$G_S(F)$ & Galois group over  $F$  of the maximal $S$-ramified extension of $F$; \\
${\mathfrak X}_F$ &  Galois group over  $F$  of the maximal $S$-ramified abelian pro-$p$-extension of  $F$; \\
& $=G_S(F)^{\textrm{ab}}\otimes \Z_p$; \\
${\mathfrak X}_v$  &   Galois group over $F_v$ of the maximal abelian pro-p-extension of $F_v$; \\
${\mathfrak X}_\infty= \varprojlim {\mathfrak X}_{F_n}$ & Galois group over  $F_\infty$  of the maximal $S$-ramified abelian pro-$p$-extension of $F_\infty$; \\
${\mathfrak X}'_F$ & maximal factor group  $({\mathfrak X}_\infty)_\Gamma$  of  ${\mathfrak X}_\infty$  on which  $\Gamma$  acts trivially; \\
$X_\infty= \varprojlim {A_n}$ & Galois group  $\gal (L_{\infty}/F_\infty)$; \\
$\Delta_\infty$  &     Gross asymptotical defect, namely,  $\varinjlim  (X_\infty^{\Gamma_n} \otimes \Q_p/\Z_p)$ ;\\
$X^\circ$  &  maximal finite submodule of  $X_\infty$. \\
\end{tabular}

When  $\mu_p$  is contained in  $F$: 

\begin{tabular}{rl}
\label{not3}

${\cal A}_F$  & the Kummer radical of the first layers of $\Z_p$-extensions of $F$;  \\
${\cal T}_F$  & subgroup of  $F^\bullet/ F^{\bullet p}$ consisting of classes  $a$ mod  $F^{\bullet p}$  such that the symbol   $\{\zeta_p, a\}=0$  in $K_2F$;\\
${\cal B}_F $  &  subgroup of  $F^\bullet/F^{\bullet p}$ consisting of classes  $a$ mod  $F^{\bullet p}$  such that 
$F (\sqrt[p]{a})/F$ \\ 
&  can be embedded in cyclic  $p$-extensions of arbitrarily large degree;  \\
\end{tabular}

The notation $( - )'$ for a group indicates that we have factored out the  $p$-primary roots of unity. Thus:

\begin{tabular}{rl}
\label{not4}

${\bar U'}_F = {\bar U}_F /\mu(F)$ 		& group of $(p)$-units in $F$ factored by its torsion part; \\

${\bar U'}_n = {\bar U}_n /\mu_{p^\infty}(F_n)$ 	& group of $(p)$-units in $F_n$ factored by its torsion part; \\ 

$U'_{\infty}=\varinjlim U'_n$    & group of $(p)$-units in $F_{\infty}$; \\ 

${\bar U}'_{\infty}= \varprojlim {\bar U}'_n$ 		& inverse limit with respect to the norm maps; \\ 

${\cal A}'_F = {\cal A}_F/(\mu(F)/p)$  &  Kummer radical obtained by factoring out the cyclotomic \\
& $\Z_p$-extension of $F$;  \\
${\cal B}'_F = {\cal B}_F/(\mu(F)/p)$  &    \\ 
${\cal T}'_F = {\cal T}_F/(\mu(F)/p)$  &    \\ 
\end{tabular}

\begin{tabular}{ll}
\label{not5}

We will also need the following notations inside the "universal kummer radical"  ${\cal K}= F_\infty^{\bullet} \otimes \Q_p/\Z_p$:\\

${\mathfrak L}\cong \hom (\fr_\Lambda ({\cal X}_\infty), \mu_{p^\infty})$; \\ 
$\mathfrak N=\varinjlim ({\bar U}_n \otimes \Q_p/\Z_p)$; \\  
$\mathfrak {\hat N}=\varinjlim ({\hat U}_n \otimes \Q_p/\Z_p)$; \\ 
$\mathfrak {\tilde N}=\varinjlim ({\tilde U}_n \otimes \Q_p/\Z_p)$; \\ 
\end{tabular}

If $n$ is a non-negative integer and $A$ is an abelian group, we 
denote by $A[n]$ the kernel of multiplication by $n$, and by $A/n$ the 
cokernel. For a prime number  $p$, we denote by  $A\{p\}$ the $p$-primary part of  $A$. 
Also $\div(A)$  will denote the maximal divisible subgroup of $A$ and  
$A/\div(A)$ is simply written  $A/\div$.
If $M$ is a module over a ring $R$, $\tor_R(M)$ is the
$R$-torsion sub-module of $M$, and $\fr_R(M) := M/\tor_R(M)$.
Finally,  $( - )^* = Hom(-, \Q_p/\Z_p)$  is the Pontrjagin dual.

\section{Norm subgroups and a radical at infinite level}
In this section, we recall (and prove if necessary) a number of results, which are fragmented and more or less well-known, concerning  some norm subgroups of  $(p)$-units and the Kummer radical of an Iwasawa module related to our problem. They will not all be needed in the sequel, but we will give as complete an account as possible, relying essentially on theorems of Kuz'min \cite {Ku72}. If pressed for time, the reader can skip this section, coming back to it if necessary.

\subsection{Global and local universal norms} 
Let  $U_F$  be the group of  $(p)$-units in our number field  $F$. 
Its pro-$p$-completion, denoted by  ${\bar U}_F$, is   $U_F \otimes \Z_p$  since  $U_F$  is finitely generated. 
Along the cyclotomic tower,  ${\bar U}_{F_n}$  is simply written ${\bar U}_n$  and we put  ${\bar U}_{\infty}= \varprojlim U_n$  for norm maps.  
The  $\Z_p$-torsion of  ${\bar U}_n$  is  the group  $\mu_{p^\infty}(F_n)$  of  $p$-primary roots of unity contained in $F_n$. By adding an apostrophe we indicate the  $\Z_p$-free part of our modules: 
$${\bar U'}_n := {\bar U}_n /\mu_{p^\infty}(F_n).$$  
Finally, we put:  ${\bar U'}_{\infty}:= \varprojlim U'_n$.

The following result, due to Kuz'min \cite{Ku72},  gives the  $\Lambda$-structure of  ${\bar U}_{\infty}$. 
It is classically known. 
\begin{prop} 
${\bar U'}_{\infty}$  is  $\Lambda$-free of rank  $r_1 + r_2$. 
When  $\mu_p \subset F$  then  we have a  $\Lambda$-module isomorphism  ${\bar U}_{\infty} \simeq \Z_p(1) \oplus {\bar U'}_{\infty}.$
\end{prop}

\begin{defn} 
The universal norm subgroup of  ${\bar U}_F$  (resp. of  ${\bar U}'_F$) is the intersection  
$\cap_{n\geq 0} N_n({\bar U}_n)$, denoted by  ${\tilde U}_F$ 
(resp. of  $\cap_{n\geq 0} N_n({\bar U}'_n)$  , denoted by  ${\tilde U}'_F$). 
Here  $N_n$ is the norm map in  $F_n/F$.  
\end{defn}

The usual compactness argument shows that  ${\tilde U}_F$  is the image of the natural map  
$$({\bar U}_{\infty})_\Gamma  \to  {\bar U}_F.$$
Furthermore,  this co-descent morphism is injective \cite [Theorem 7.3] {Ku72}  
so that  ${\tilde U}_F  \simeq  ({\bar U}_{\infty})_\Gamma$  is of  $\Z_p$-rank equal to $r_1 + r_2$.  
Similarly,  ${\tilde U}'_F$  is isomorphic to  $({\bar U'}_{\infty})_\Gamma$  and is of  $\Z_p$-rank  $r_1 + r_2$.  

We are now going to define a second norm subgroup of  ${\bar U}_F$.  
Let $\bar{F}_v^{\bullet} := \varprojlim  F_v^{\bullet} /F_v^{\bullet p^n}$  denote the pro-$p$-completion of  $F_v^{\bullet}$,  and  
$\tilde{F}_v^{\bullet}$  denote the group of universal norms in the cyclotomic $\Z_p$-extension of $F_v$ (the local analogue of definition 1.2).  
Let  $X_\infty := \gal(L_\infty/F_\infty)$, where $L_\infty$  is the maximal unramified abelian pro-$p$-extension of  $F_{\infty}$ in which every prime over $p$ (hence every prime) splits.    
We have the following Sinnott exact sequence (see \cite[etc]{FGS81, Ja87, Ko91}):
$$ {\bar U}_F \overset{g_F} \longrightarrow {\tilde \oplus}_{v|p}\bar{F}_v^{\bullet}/\tilde{F}_v^{\bullet} \overset{Artin} \longrightarrow  (X_\infty)_\Gamma \to A_F  \to 0  \quad  (Sinnott)$$
where  $g_F$  is naturally deduced from the diagonal map  taking its values in 
${\tilde \oplus}_{v|p}\bar{F}_v^{\bullet}/\tilde{F}_v^{\bullet}$, 
the elements having a trivial sum of components. 

\begin{defn} 
The everywhere local universal norm subgroup  ${\hat U}_F$  of  ${\bar U}_F$  consists of those elements which are locally universal norms in the cyclotomic  $\Z_p$-extensions  $F_{\infty,v}/F_v$  for all finite primes  $v$  in  $F$. 
Since for a non-$p$-adic prime $v$, the cyclotomic  $\Z_p$-extension  $F_v$  is unramified, the group of universal  norms in $F_{\infty,v}/F_v$ is precisely the group of units of  $F_v$. Therefore 
$${\hat U}_F = \ker g_F.$$
\end{defn}

We obviously have the following inclusions  ${\tilde U}_F  \subset  {\hat U}_F  \subset {\bar U}_F,$  as well as  
${\tilde U}'_F  \subset  {\hat U}'_F  \subset {\bar U}'_F,$  where  an apostrophe indicates factoring out the  $\Z_p$-torsion subgroup  $\mu(F).$ 

The next lemma, which is proved by class field theory, gives a characterization of the Gross kernel  ${\hat U}_F$ (see \cite  [Proposition 2.1]{BP72} and \cite [section 1] {Ko91}). 
\begin{lemma} 
For an element  $x \in {\bar U}_F$, we have: 
$$x \in {\hat U}_F  \iff  N_{F_v/\Q_p}(x) \in  p^{\Z} \quad \forall v|p.$$
\end{lemma} 

This lemma implies that  ${\hat U}_F$  is  "accessible"  in the sense of the introduction. 
According to the above Sinnott exact sequence,  the $\Z_p$-rank of  ${\hat U}_F$  is equal to  $r_1+r_2+\delta$, where   
$\delta :=rk_{\Z_p} (X_\infty)_\Gamma$. The number field  $F$  satisfies Gross' (generalized) Conjecture at $p$ if  $\delta = 0$, namely if  
$ (X_\infty)_\Gamma$  (or equivalently  $X_\infty^\Gamma$, because  $X_\infty$  is  $\Lambda$-torsion)   is finite.  Gross' Conjecture is known to hold for abelian extensions of  $\Q$  \cite{Gre73, Ja87}.  
In the inclusion tower  ${\tilde U}_F  \subset  {\hat U}_F  \subset {\bar U}_F,$  we have  
$${\bar U}_F/{\hat U}_F \simeq \im g_F \simeq \Z_p^{s-1-\delta}$$  
where  $s$  is the number of  $p$-adic primes of  $F$.  
Concerning the deviation between  ${\hat U}_F$  and  ${\tilde U}_F$  we have the following local-global result: 

\begin{prop} (\cite [Proposition 7.5] {Ku72})
There exists a canonical exact sequence 
$$ 0 \to {\tilde U}_F  \to {\hat U}_F \to X_\infty^\Gamma \to 0.$$
\end{prop}

Kuz'min's proof is class field theoretic. For a proof with a more Iwasawa-theoretic flavour, see 
\cite [Theorem 3.3] {KNF96}. See also \cite [Section 3.2] {Ka06}  for a cohomological proof. 
Our interest in  ${\tilde U}_F$,  as we recalled in the introduction, lies in the fact that, in the Kummerian situation, every element of  ${\tilde U}_F$  starts a  $\Z_p$-extension. We give a quick proof for the convenience of the reader: 

\begin{lemma} 
When  $\mu_p \subset F$, the Kummer radical  ${\cal A}_F$  contains  ${\tilde U}_F F^{\bullet p}/F^{\bullet p}$. 
\end{lemma} 

\begin{proof} 
Recall Kuz'min's result that the natural codescent map gives an isomorphism  
$({\bar U}_\infty)_\Gamma \simeq {\tilde U}_F \subset {\bar U}_F \simeq H^1(G_S(F), \Z_p(1))$. 

On the other hand,  codescent on the (-1)-twist gives a homomorphism  
${\bar U}_\infty(-1)_\Gamma  \rightarrow \hom(G_S(F), \Z_p)$ (see \cite [Theorem 3.7]{FN91}, 
where  $NB({\cal R}_F, \Zp) \simeq {\bar U}_\infty(-1)_\Gamma$   and  
$G({\cal R}_F, \Zp) \simeq \hom(G_S(F), \Z_p)$). 
Hence a composite homomorphism   
$${\bar U}_\infty(-1)_\Gamma/p  \rightarrow \hom(G_S(F), \Z_p)/p  \stackrel{nat} \longrightarrow \hom(G_S(F), \Z/p).$$

Since  ${\bar U}_\infty(-1)_\Gamma/p = ({\bar U}_\infty)_\Gamma/p (-1) \simeq {\tilde U}_F/p(-1)$,  it follows from Proposition 2.3, op. cit. that the induced map  
$$({\tilde U}_F/p)(-1) \longrightarrow \hom(G_S(F), \Z/p) = \hom(G_S(F), \mu_p)(-1)$$  
is just the  (-1)-twist of  
${\tilde U}_F/p  \longrightarrow {\bar U}_F/p$.  As it factors through   $\hom(G_S(F), \Z_p)/p(-1)$, the proof is complete. 
\end{proof}

\subsection{A Kummer radical at infinite level} 
Let  $\cal X_\infty$  be the Galois group over  $F_\infty$  of the maximal abelian $p$-extension of  $F_\infty$  which is unramified outside  $p$-adic primes. It is known that the  $\Lambda$-rank of  $\cal X_\infty$  is equal to  $r_2$  (this is the weak Leopoldt conjecture, which holds in the case of the cyclotomic  $\Z_p$-extension   \cite [Section 13.5] {Wa97}).  
Put  $\fr_\Lambda (\cal X_\infty)$  for its torsion-free part.  
When  $\mu_p \subset F$, the Kummer radical of  
$\fr_\Lambda (\cal X_\infty)$ is clearly related to the problem we are interested in. The determination of this Kummer radical has been performed independently by  Kuz'min (\cite {Ku72})  and Kolster  (\cite {Ko91})  using 
idelic methods. Here, we are going to prove a slightly more precise version of their result using a direct approach. 

First fix  the following notations:  
${\mathfrak L}\cong \hom (\fr_\Lambda ({\cal X}_\infty), \mu_{p^\infty}) \subseteq F_\infty^{\bullet} \otimes \Q_p/\Z_p$,  \\
$\mathfrak N = \varinjlim ({\bar U}_n \otimes \Q_p/\Z_p), \quad  
\mathfrak {\hat N} = \varinjlim ({\hat U}_n \otimes \Q_p/\Z_p), \quad
\mathfrak {\tilde N} = \varinjlim ({\tilde U}_n \otimes \Q_p/\Z_p). \quad$

For all  $n \geq 0$, define  $V_n$  to fit into the following tautological short exact sequence 
$$ 0 \to {\tilde U}_n  \to {\bar U}_n \to V_n \to 0.$$
We then have a commutative diagram 
$$\xymatrix{
0 \ar[r] & {\tilde U}_n \ar[r] \ar[d]^{=} &  {\hat U}_n  \ar[r] \ar@{^{(}->}[d]  & X_\infty^{\Gamma_n} \ar[r]  & 0 \\
0 \ar[r] & {\tilde U}_n \ar[r]  &  {\bar U}_n \ar[r] \ar[d]_{} & V_n \ar[r]  & 0 \\
 &  & {\tilde \oplus}_{v|p}\bar{F}_v^{\bullet}/\tilde{F}_v^{\bullet} \ar[d]  \simeq \Z_p^{s_n-1}  &  &  \\
 &  &  (X_\infty)_{\Gamma_n}  \ar[d]_{} &  &  \\
 &  &  A_n  &  &  \\
}$$
where  $s_n$  is the number of  $p$-adic primes in the layer  $F_n$.   
This immediately provides a short exact sequence 
$$0 \longrightarrow  X_\infty^{\Gamma_n} \longrightarrow V_n  \longrightarrow \Z_p^{s_n-1-\delta_n}  \to 0$$ 
where  $\delta_n :=rk_{\Z_p} (X_\infty)_{\Gamma_n}.$
In particular 
$\tor_{\Z_p}(V_n) \simeq \tor_{\Z_p}(X_\infty^{\Gamma_n})$. 

Note that  $\varinjlim \tor_{\Z_p}(X_\infty^{\Gamma_n}) = X^\circ$  is the maximal finite submodule of  $X_\infty$  
and the groups   $X_\infty^{\Gamma_n} \otimes \Q_p/\Z_p$  stabilize. 
Put  $\Delta_\infty:= \varinjlim  (X_\infty^{\Gamma_n} \otimes \Q_p/\Z_p)$  for the "Gross asymptotical defect". 

\begin{prop} (\cite{Ko91, Ku72}) 
We have an exact sequence of  $\Gamma$-modules
$$\xymatrix{
0 \ar[r] & X^\circ  \ar[r] &  {\mathfrak {\tilde N}}  \ar[rr] \ar@{^{}->>}[rd] & & {\mathfrak {\hat N}} \ar[r] & 
 \Delta_\infty  \ar[r] & 0 \\
  &   &   &  {\mathfrak L} \ar@{^{(}->}[ru] &  & &\\
}$$ 
In particular,  ${\mathfrak L} = {\mathfrak {\hat N}}$  precisely when all the layers  $F_n$  verify Gross' conjecture. 
\end{prop}

\begin{proof}
Tensoring with  $\Q_p/\Z_p$  the exact sequence  
$0 \to {\tilde U}_n  \to {\hat U}_n \to X_\infty^{\Gamma_n} \to 0$  
of Proposition 1.5   we get: 
$$0 \to \tor_{\Z_p}(X_\infty^{\Gamma_n}) \to {\tilde U}_n \otimes \Q_p/\Z_p  \to {\hat U}_n \otimes \Q_p/\Z_p 
\to X_\infty^{\Gamma_n} \otimes \Q_p/\Z_p \to 0.$$ 
Hence, passing to the direct limit over  $n$  yields  
$$0 \to X^\circ  \to  {\mathfrak {\tilde N}} \to {\mathfrak {\hat N}} \to \Delta_\infty  \to 0$$
where  $X^\circ$  is the maximal finite submodule of  $X_\infty$. 
Now it remains to show that the cokernel of the first map  $ X^\circ  \to  {\mathfrak {\tilde N}}$  is precisely  ${\mathfrak L}$. 
Proceeding in the same way as before from the exact sequence 
$$ 0 \to {\tilde U}_n  \to {\bar U}_n \to V_n \to 0$$
we get: 
$$0 \to X^\circ  \to  {\mathfrak {\tilde N}} \to {\mathfrak N} \to 
{\mathfrak V}:= \varinjlim (V_n \otimes \Q_p/\Z_p)  \to 0$$
which we split into two short exact sequences: 
$$0 \to X^\circ  \to  {\mathfrak {\tilde N}} \to {\mathfrak L'} \to 0  
\quad \textrm { and } \quad 
0 \to {\mathfrak L'} \to {\mathfrak N} \to {\mathfrak V}  \to 0.$$
Write the Kummer dual of the first sequence: 
$$0 \to \hom ({\mathfrak L'}, \mu_{p^\infty}) \to \hom ({\mathfrak {\tilde N}}, \mu_{p^\infty}) 
\to \hom (X^\circ, \mu_{p^\infty}) \to 0.$$
By definition, ${\tilde U}_n = ({\overline U}_\infty)_{\Gamma_n}$. 
In view of the properties of   ${\overline U}_\infty$  (Proposition 1.1), the module  ${\mathfrak {\tilde N}}$  is a co-free module over  
$\Lambda$ of co-rank  $r_2$:  $\hom ({\mathfrak {\tilde N}}, \mu_{p^\infty}) \simeq \Lambda^{r_2}$. 
Therefore  $\hom ({\mathfrak L'}, \mu_{p^\infty})$  is  $\Lambda$-torsion-free of the same rank  $r_2$. 
Now, the Kummer dual of the second sequence: 
$$0 \to \hom ({\mathfrak V}, \mu_{p^\infty}) \to \hom ({\mathfrak N}, \mu_{p^\infty}) 
\to \hom ({\mathfrak L'}, \mu_{p^\infty}) \to 0$$  
shows that  $\hom ({\mathfrak L'}, \mu_{p^\infty})$  is a quotient module of  $\hom ({\mathfrak N}, \mu_{p^\infty})$, and it is known that  $\fr_{\Lambda} {\mathfrak X}_\infty$  is the maximal  $\Lambda$-torsion free quotient of the Galois group  $\hom ({\mathfrak N}, \mu_{p^\infty})$  (\cite [Theorem 15 ]{Iw73}). Accordingly, we have a surjective map  
$\fr_{\Lambda} {\mathfrak X}_\infty \to  \hom ({\mathfrak L'}, \mu_{p^\infty})$,   
which must be an isomorphism since they both have    $\Lambda$-rank  $r_2$. 
Hence  ${\mathfrak L'} = {\mathfrak L}$,  as was to be shown. 
\end{proof}

\textbf{Remark}: 
The structure theorem for finitely generated  $\Lambda$-modules shows the existence of a finite module  $H$ such that  
$$ 0 \to \fr_{\Lambda} {\mathfrak X}_\infty  \to \Lambda^{r_2} \to H \to 0.$$ 
The above proof provides such an exact sequence in a canonical way as well as an isomorphism  
$H \simeq \hom (X^\circ, \mu_{p^\infty}).$ This Kummer duality between  $H$  and  $X^\circ$  was already implicit in  \cite {Iw73}. It is also known (op. cit.) that  $X^\circ$  is isomorphic to the (asymptotical) capitulation kernel  $\ker (A_m \to \varinjlim A_n)$  for  $m$  large  \cite [page 1240 ] {Gre78}.

\section{Lower bounds for the Kummer radical  ${\cal A}_F$} 
In this section we suppose that  $F$  contains  $\mu_p$.  
In order to get information about the radical  ${\cal A}_F$, we are first going to do "descent" from the module  ${\mathfrak L}$  in the same way as in \cite {Gre78}  (but we will need somewhat more precise results). 
Recall the notations : 
${\tilde U}'_F = {\tilde U}_F/\mu(F)$, ${\hat U}'_F = {\hat U}_F/\mu(F)$  and let  
${\cal A'}_F := {\cal A}_F/(\mu(F)/p)$,  where   $\mu(F)/p \simeq \mu(F) F^{\bullet p}/ F^{\bullet p}$  
is the Kummer radical of the first layer of the cyclotomic  $\Z_p$-extension of $F$.
 
\subsection{A lower bound in terms of universal norms} 
The starting point will be the exact sequence 
$$0 \to X^\circ  \to  {\mathfrak {\tilde N}} \to {\mathfrak L} = {\mathfrak L}' \to 0$$
 of Proposition 1.7, where we recall that  
$\mathfrak {\tilde N} = \varinjlim ({\tilde U}_n \otimes \Q_p/\Z_p)$
and 
${\mathfrak L} \cong \hom (\fr_\Lambda ({\cal X}_\infty), \mu_{p^\infty}) \subseteq F_\infty^{\bullet} \otimes \Q_p/\Z_p$.

For all  rational integers  $i$, we consider the  $i$-fold Tate-twists: 
\begin{equation} 
0 \to X^\circ(i)  \to  {\mathfrak {\tilde N}}(i) \to {\mathfrak L}(i) \to 0. 
\end{equation} 

\begin{thm}
For  $i \in \Z$, we have an exact sequence: 
$$0 \to  X^{\circ \Gamma}[p](i) \to  {\tilde U}'_F/p(i)  \to  
\div({\mathfrak L}(i)^\Gamma)[p] \to X^\circ(i)^\Gamma/p \to  0$$
where $\div(-)$  denotes the maximal divisible subgroup of  $(-).$  
\end{thm}

Notice that, since  $\mu_p \subset F$,  the twist  $i$  outside is purely cosmetic for the Galois action above  $F$, but of course not for the action below  $F$. 

\begin{proof}
The exact sequence (1) provides us with: 
$$0 \to X^\circ(i)^\Gamma  \to  {\mathfrak {\tilde N}}(i)^\Gamma \to {\mathfrak L}(i)^\Gamma \to 
X^\circ(i)_\Gamma  \to  {\mathfrak {\tilde N}}(i)_\Gamma \to \cdots $$
Let  $N_i$  be the cokernel of the first map on the left: 
$0 \to X^\circ(i)^\Gamma  \to  {\mathfrak {\tilde N}}(i)^\Gamma \to N_i \to 0$. 
As noticed before (Proposition 1.1), the  $\Lambda$-module  ${\mathfrak {\tilde N}}$  is cofree. 
Hence  ${\mathfrak {\tilde N}}(i)_\Gamma$  is trivial whereas  ${\mathfrak {\tilde N}}(i)^\Gamma$  is divisible and therefore   
$N_i = \div({\mathfrak L}(i)^\Gamma)$. 
Consequently, we have: 
$$0 \to X^\circ(i)^\Gamma  \to  {\mathfrak {\tilde N}}(i)^\Gamma \to \div ({\mathfrak L}(i)^\Gamma) \to 0.$$
Applying the snake lemma to multiplication-by-$p$ gives  
$$0 \to {X^\circ(i)}^\Gamma[p]={X^\circ}^\Gamma[p](i)  \to  
{\mathfrak {\tilde N}}(i)^\Gamma[p]={\mathfrak {\tilde N}}^\Gamma[p](i) \to 
\div ({\mathfrak L}(i)^\Gamma)[p] \to {X^\circ(i)}^\Gamma/p \to 0.$$
To finish the proof of the Theorem, it remains to recognize  ${\mathfrak {\tilde N}}^\Gamma$: 

\begin{lemma} 
${\tilde U}'_F \otimes \Q_p/\Z_p \simeq {\mathfrak {\tilde N}}^\Gamma$. 
\end{lemma} 
{\it Proof.} 
At finite levels, we have  ${\hat U}'_n \hookrightarrow {\hat U}'_{n+1}$  (this is immediate from Lemma 1.4). 
Hence we also have  ${\tilde U}'_n \hookrightarrow {\tilde U}'_{n+1}$. 
By definition 1.2, the norm maps  ${\tilde U}'_n \to {\tilde U}'_F$  are surjective. 
On the other hand, by Proposition 1.1, all the  ${\tilde U}'_n$ are  free $\Z_p$-modules.  
Hence  ${\tilde U}'_F$  is a direct summand of  ${\tilde U}'_n$   and the cokernel of  the natural injection  ${\tilde U}'_F \to {\tilde U}'_n$  is torsion-free. Accordingly, the maps 
 ${\tilde U}'_F \otimes \Q_p/\Z_p  \to  {\tilde U}'_n \otimes \Q_p/\Z_p$  
 are injective and so is the map:  
 ${\tilde U}'_F \otimes \Q_p/\Z_p \hookrightarrow {\mathfrak {\tilde N}}^\Gamma$. 

Moreover, as previously explained,  ${\mathfrak {\tilde N}}^\Gamma$  is divisible and to show the equality all we need to do is compare the  $\Z_p$-coranks of these two modules: the exact sequence  
$$0 \to {X^\circ}^\Gamma  \to  {\mathfrak {\tilde N}}^\Gamma \to \div ({\mathfrak L}^\Gamma) \to 0$$ 
shows that  $\corank_{\Z_p} ({\mathfrak {\tilde N}}^\Gamma) = r_2$.  On the other hand: 
$$\begin{array}{llll} 
\rank_{\Z_p}({\tilde U}'_F) & = & \rank_{\Z_p}({\hat U}'_F) - \rank_{\Z_p}(X_\infty)^\Gamma & \quad (Kuz'min) \\
 & = & \rank_{\Z_p}({\bar U}'_F) - (s-1) +  \rank_{\Z_p}({X_\infty})_\Gamma -\rank_{\Z_p}(X_\infty)^\Gamma & \quad (Sinnott) \\  
 & = & r_2 +  \rank_{\Z_p}{(X_\infty})_\Gamma -\rank_{\Z_p}(X_\infty)^\Gamma,&  \\  
\end{array}$$ 
where  $s$  is the number of $p$-adic primes in  $F$. 
Besides, since  $X_\infty$  is a  $\Lambda$-torsion module, denoting by  $Y_\infty$  the "non-$\mu$-part" of  $X_\infty$, we have a "pseudo-exact" sequence  
$$0 \to {(X_\infty})^\Gamma  \to Y_\infty \stackrel{\gamma-1} \longrightarrow Y_\infty \to  {(X_\infty})_\Gamma \to 0$$  
of finitely generated  $\Z_p$-modules which shows that  ${(X_\infty})^\Gamma$  and  ${(X_\infty})_\Gamma$  have the same   $\Z_p$-rank. Hence finally 
the   $\Z_p$-corank of  ${\tilde U}'_F \otimes \Q_p/\Z_p$  is also  $r_2$. 
\end{proof}

We now propose to give interpretations of  $\div({\mathfrak L}(i)^\Gamma)[p]$  in terms of radicals in the three special cases  $i=-1, 0, 1$.  We have already introduced   
${\cal A}'_F := {\cal A}_F/(\mu(F)/p)$  and   
$({\hat U}'_F/p) = ({\hat U}_F/p)/(\mu(F)/p)$, 
where  $\mu(F)/p \cong  \mu(F) F^{\bullet p}/F^{\bullet p}$  is the image of  
$\mu_{p^\infty}$  in  $F^{\bullet}/F^{\bullet p}.$  
We also want to introduce a modified tate kernel  ${\cal T}'_F$. Using elementary properties of symbols, one easily shows that  $\mu(F) F^{\bullet p}/F^{\bullet p}$  is also contained in  ${\cal T}_F$, thus we define  ${\cal T}'_F := {\cal T}_F/(\mu(F)/p)$.

\begin{prop}   We have the following equalities:\\
(a) if  $i=-1$ and we assume Leopoldt's conjecture for  $F$, then $\div({\mathfrak L}(-1)^\Gamma)[p](1)={\cal A}'_F$. \\
(b) if  $i=0 $ and we assume Gross' conjecture for  $F$, then $\div({\mathfrak L}^\Gamma)[p]={\hat U}'_F/p$. \\
(c) if  $i=1,$ then   $\div({\mathfrak L}(1)^\Gamma)[p](-1)={\cal T}'_F$. 
\end{prop}

Recall that if Gross' conjecture is valid for all the $F_n$'s, then  ${\mathfrak L}={\mathfrak {\hat N}}$.  

{\it Proof.}
Properties (a)  and  (c)  have been explained by Greenberg, starting from his conjecture  
(\cite [page 1238]{Gre78}; see also \cite[page 192]{Sc79})  that for all  $i \neq 0$,  the co-rank of  
$\div({\cal K}(i)^\Gamma)$  should be  $r_2$,  hence  $\div({\cal K}(i)^\Gamma) = \div({\mathfrak L}(i)^\Gamma)$  where  ${\cal K}:=F_\infty^{\bullet}\otimes \QpZp \cong H^1(F_\infty, \mu_{p^\infty})$. \\ 
For  $i=-1$ this is Leopoldt's conjecture and for  $i=1$  a consequence of results of Tate on  $K_2$ 
(see \cite [Theorem 2]{Co72}). 
For the convenience of the reader, we reprove  (a)  and  (c)  along our own lines. 
For  (a), 
let  ${\hat F}$  be the maximal abelian  $p$-extension of  $F$  which is unramified outside  $p$-adic primes and  ${\mathfrak X}_F : = \gal({\hat F}/F)$. 
Put  ${\mathfrak X}'_F := ({\mathfrak X}_\infty)_\Gamma = \gal({\hat F}/F_\infty)$. 
Then, the following exact sequence of co-descent 
$$ 0 \to \mathfrak{X}'_F \to \mathfrak{X}_F \to \gal(F_\infty/F) \simeq \Z_p \to 0 $$ 
shows, under Leopoldt's conjecture, that  $\tor_{\Z_p} (\mathfrak{X}'_F) = \tor_{\Z_p} (\mathfrak{X}_F)$  and the exactness of  
$$ 0 \to \fr_{\Z_p} \mathfrak{X}'_F \to \fr_{\Z_p} \mathfrak{X}_F \to  \Z_p \to 0.$$ 
Consequently, our  ${\cal A}'_F$  is the Kummer radical of   $(\fr_{\Z_p} \mathfrak{X}'_F)/p$. 
Besides, the co-descent from  
$$ 0 \to \tor_{\Lambda} \mathfrak{X}_\infty \to \mathfrak{X}_\infty \to \fr_{\Lambda} \mathfrak{X}_\infty \to 0$$  
yields a commutative diagram 
$$ \xymatrix{
0 \ar[r] &  (\tor_{\Lambda} \mathfrak{X}_\infty)_\Gamma  \ar[r]  \ar@{^{(}->}[d]^{\epsilon_1} &  
(\mathfrak{X}_\infty)_\Gamma = \mathfrak{X}'_F  \ar[r] \ar@{=}[d] & 
(\fr_{\Lambda} \mathfrak{X}_\infty)_\Gamma  \ar[r] \ar@{^{ }.>>}[d]^{\epsilon_2}& 0 \\
0 \ar[r] & \tor_{\Z_p} (\mathfrak{X}_F) \ar[r]  &  \mathfrak{X}'_F  \ar[r] & 
\fr_{\Z_p} \mathfrak{X}'_F \ar[r]   & 0 \\
}$$

Leopoldt's conjecture for $F$ is equivalent to the finiteness of   $(\tor_{\Lambda} \mathfrak{X}_\infty)_\Gamma$. 
This ensures that the vertical map  $\epsilon_1$  takes its values in  $\tor_{\Z_p} (\mathfrak{X}_F)$. 
Hence the existence of the map  $\epsilon_2$ and the isomorphism  $\ker \epsilon_2 \simeq \coker \epsilon_1$. 
In particular,  $\ker \epsilon_2$  is finite  and therefore   
$\fr_{\Z_p} ((\fr_{\Lambda} \mathfrak{X}_\infty)_\Gamma) \simeq \fr_{\Z_p}(\mathfrak{X}'_F).$ 

Write the exact sequence of the Remark following Proposition 1.7: 
$$ 0 \to \fr_{\Lambda} {\mathfrak X}_\infty  \to \Lambda^{r_2} \to H \to 0$$ 
where  the finite module  $H$ is isomorphic to  $\hom(X^\circ, \mu_{p^\infty})$. 
By codescent, we obtain
$$ 0 \to H^\Gamma \to (\fr_{\Lambda} {\mathfrak X}_\infty)_\Gamma  \to \Z_p^{r_2}$$ 
which shows that: $H^\Gamma \simeq \tor_{\Z_p} ( (\fr_{\Lambda} {\mathfrak X}_\infty)_\Gamma )$, hence the following exact sequence 
$$ 0 \to H^\Gamma \to (\fr_{\Lambda} {\mathfrak X}_\infty)_\Gamma  \stackrel{\epsilon_2} \to \fr_{\Z_p}(\mathfrak{X}'_F) \to 0.$$ 
The situation is illustrated in the following diagram of fields and Galois groups: 
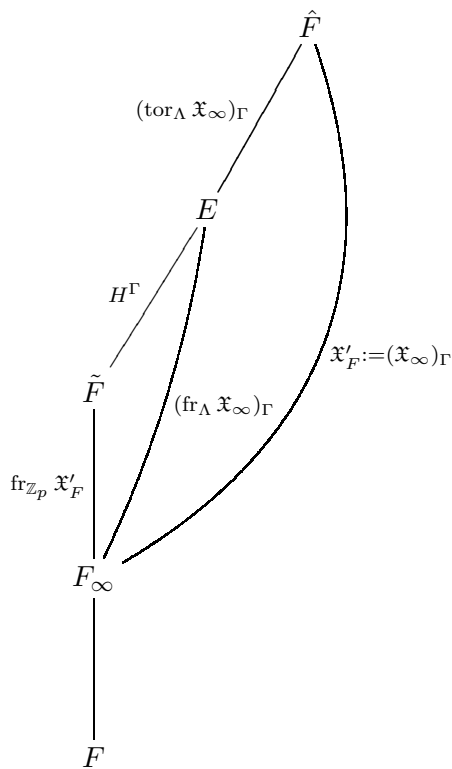
\begin{displaymath}
\xymatrix{
& & & {\hat F} \ar@/^4pc/@{-}[6,-2]^{\mathfrak{X}'_F := (\mathfrak{X}_\infty)_\Gamma} \ar@{-}[2,-1]\ar@{-}_{ (\tor_{\Lambda} \mathfrak{X}_\infty)_\Gamma }[2,-1]  \\
\\
& & E \ar@{-}_{H^\Gamma}[2,-1]\ar@/^/@{-}^{(\fr_{\Lambda} {\mathfrak X}_\infty)_\Gamma}[4,-1]  \\
\\
&{\tilde F} \ar@{-}_{\fr_{\Z_p} \mathfrak{X}'_F}[2,0]\\
\\
&F_\infty \ar@{-}[2,0]\\
\\
& F  
}
\end{displaymath}
The above short exact sequence gives by Kummer duality 
$$0 \to \hom (\fr_{\Z_p}\mathfrak{X}'_F, \mu_{p^\infty}) \to  {\mathfrak L}(-1)^\Gamma (1) \to (X^\circ(-1)_\Gamma)(1) \to 0$$  
where, as before,  ${\mathfrak L} = \hom (\fr_{\Lambda} \mathfrak{X}_\infty, \mu_{p^\infty}).$  
This shows that the first term is  $\div ({\mathfrak L}(-1)^\Gamma)(1)$  and the snake lemma gives the exact sequence 
$$ 0 \to {\cal A}'_F \to  {\mathfrak L}(-1)^\Gamma [p](1) = {\mathfrak L}^\Gamma [p]  \to  (X^\circ(-1)_\Gamma)[p](1) \to 0$$ 
(a precise version of the result in \cite [page 1242] {Gre78}). 
This shows in particular that under Leopoldt's conjecture  ${\cal A}'_F = \div({\mathfrak L}(-1)^\Gamma)[p](1)$.

 To show  (c), start from the commutative diagram 
$$\xymatrix{
 &   &   & 0  \ar[d] &  \\
 &   &   & \div ({\cal K}(1)^\Gamma)[p]  \ar[d] &  \\
0 \ar[r] & (\mu(F)/p)(1) \ar[r] \ar@{^{(}->}[d] &  F^{\bullet}/F^{\bullet p}(1)  \ar[r]  \ar[d]^= & {\cal K}(1)^\Gamma[p] = {\cal K}^\Gamma[p] (1) \ar[r] \ar[d]^\beta & 0 \\
0 \ar[r] & {\cal T}_F(1) \ar[r]  &  F^{\bullet}/F^{\bullet p}(1) \ar[r]^\alpha  & K_2F[p] \ar[r]  \ar[d]& 0 \\
 &   &   & 0   &  \\
}$$
The upper exact line comes from Kummer theory: just write the cohomology exact sequence  
$$0 \to H^0(F, \mu_{p^\infty})/p \to  H^1(F, \mu_p) \to H^1(F, \mu_{p^\infty})[p] \to 0 $$
and use the nullity of  
$H^1(\Gamma, \mu_{p^\infty})$  (Tate's lemma)  to show that  
$$H^1(F, \mu_{p^\infty}) \cong H^1(F_\infty, \mu_{p^\infty})^\Gamma \cong {\cal K}^\Gamma.$$
The map  $\alpha$  in the lower exact line is Tate's map  $a\otimes \zeta_p  \mapsto \{a,\zeta_p\}$  which defines the Tate kernel  ${\cal T}_F$. The exact column is derived from  Tate's exact sequence 
(see e.g. \cite [Theorem 3]{Co72})
$$0 \to \div ({\cal K}(1)^\Gamma) \to  {\cal K}(1)^\Gamma  \stackrel{\beta} \to K_2F\{p\} \to 0 $$
where  $\beta$  is the map  $a\otimes \zeta  \mapsto \{a,\zeta\}$,  for  $a \in F^{\bullet}_\infty$  and   
$\zeta \in \mu_{p^\infty}$.  
A simple diagram chase now gives (c). 

Greenberg's conjecture does not concern the twist  $i=0$, since  ${\cal K}^\Gamma$  is of infinite co-type. 
Hence the assertion  (b)  requires special treatment:  

\begin{lemma} 
Assuming Gross' conjecture for  $F$, we have an exact sequence 
$$0 \to {\hat U}'_F \otimes \Q_p/\Z_p \to {\mathfrak L}^\Gamma \to (X^\circ)_\Gamma \to 0.$$
In particular, 
${\hat U}'_F \otimes \Q_p/\Z_p = \div({\mathfrak L}^\Gamma)$. 
\end{lemma} 
\begin{proof} 
We have two short exact sequences ( Proposition 1.7 and its proof)  
$$0 \to X^\circ  \to  {\mathfrak {\tilde N}} \to {\mathfrak L} \to 0  
\quad \textrm { and } \quad 
0 \to {\mathfrak L} \to {\mathfrak {\hat N}} \to \Delta_\infty  \to 0.$$
From the first, we have already derived: 
$$0 \to (X^\circ)^\Gamma  \to  {\mathfrak {\tilde N}}^\Gamma \to {\mathfrak L}^\Gamma \to (X^\circ)_\Gamma \to 0.$$
Since  ${\mathfrak L}$  is the Kummer dual of  $\fr_\Lambda (\mathfrak{X}_\infty)$, we have  ${\mathfrak L}_\Gamma = 0$ so that the second sequence implies: 
$$0 \to {\mathfrak L}^\Gamma \to {\mathfrak {\hat N}}^\Gamma \to \Delta_\infty^\Gamma \to 0 .$$
We compare this exact sequence to what we have at the level of  $F$ which is provided by Proposition 1.5 after tensoring with  $\Q_p/\Z_p$: 
$$ 0 \to \tor_{\Z_p}(X_\infty^\Gamma) \to {\tilde U}'_F\otimes \Q_p/\Z_p  \to {\hat U}'_F\otimes \Q_p/\Z_p \to X_\infty^\Gamma\otimes \Q_p/\Z_p \to 0.$$ 
Namely, we break the above exact sequence into two exact sequences to obtain a commutative diagram: 
$$\xymatrix{
0 \ar[r] &  (X^\circ)^\Gamma \ar[r]  &  {\mathfrak {\tilde N}}^\Gamma  \ar[r] & {\mathfrak L}^\Gamma \ar[r] &  
(X^\circ)_\Gamma   \ar[r]  &  0 \\
0 \ar[r] &\tor_{\Z_p}(X_\infty^\Gamma) \ar[r] \ar@{=}[u] &  {\tilde U}'_F\otimes \Q_p/\Z_p \ar[r] \ar[u]^{\wr}& M \ar[r]  \ar[u] & 0 \\
}$$
leading to  
$$ 0 \to M \to {\mathfrak L}^\Gamma \to (X^\circ)_\Gamma \to 0.$$ 
To finish the proof, observe that under the Gross conjecture for  $F$,  the tensor product  
$X_\infty^\Gamma\otimes \Q_p/\Z_p$  vanishes since  $X_\infty^\Gamma$ is finite. 
Hence in fact  $M ={\hat U}'_F\otimes \Q_p/\Z_p$. 
\end{proof}
Since we are specifically dealing with the case (a) of the above Proposition, let us restate it after twisting once "\`a la Tate". 
\begin{cor}
If  $F$  verifies Leopoldt's conjecture at $p$, then we have a short exact sequence   
$$0  \to  {\tilde U}_FF^{\bullet p}/F^{\bullet p} \to {\cal A}_F \to ({X^\circ(-1)}^\Gamma/p)(1)  \to 0. $$
In particular, ${\cal A}_F = {\tilde U}_F F^{\bullet p}/F^{\bullet p}$  precisely when 
$X^\circ = 0$. 
\end{cor}
\begin{proof} 
The case  $i=-1$  of Theorem 2.1  yields: 
$$0  \to  {\tilde U}'_F F^{\bullet p}/F^{\bullet p} \to 
{\cal A}'_F \to ({X^\circ(-1)}^\Gamma/p)(1) = ({X_\infty(-1)}^\Gamma/p)(1) \to 0$$
(The last equality comes from Leopoldt's conjecture, which is known to be equivalent to the finiteness of  ${X_\infty(-1)}^\Gamma$). 
Since  ${\tilde U}'_F F^{\bullet p}/F^{\bullet p}$  and   ${\cal A}'_F$  are obtained by taking the quotients of  
${\tilde U}_F F^{\bullet p}/F^{\bullet p}$  and   ${\cal A}_F$ by the same submodule  $\mu(F)/p$, the exact sequence of the corollary follows. 
\end{proof} 

\begin{ex}
Consider a number field  $F$ which satisfies Leopoldt's conjecture together with the following two properties: \\
(i) $X^\circ =(0)$  and  (ii) $F$ contains only one  $p$-adic prime \\
(for a cyclotomic field  $F:=\Q(\mu_{p^r})$, the property (i) is equivalent to Vandiver's conjecture and (ii) is of course automatically satisfied). 
Then  ${\cal A}_F = {\hat U}_F F^{\bullet p}/F^{\bullet p}$ according to Proposition 1.5 and Corollary 2.5, and  ${\overline U}_F ={\hat U}_F$  according to (ii).  
It follows that  ${\cal A}_F = {\overline U}_F/p$.  
\end{ex}

\textbf{Remarks}:
(i) Lemma 2.4 was shown in \cite {Hu05}  using a different approach. \\
(ii) Within the context of assertion (b), the exact sequence of Theorem 2.1 coincides with the one obtained by applying the snake lemma to multiplication-by-$p$ in  Kuz'min's exact sequence  (Proposition 1.5). \\
(iii) For  $i \neq 0$, the exact sequence of Theorem 2.1 comes also by applying the snake lemma to multiplication-by-$p$ in some descent exact sequences in Galois cohomology \cite [Theorem 3.2 bis]{KNF96}. \\
(iv) In Corollary 2.5,  ${\tilde U}'_F F^{\bullet p}/F^{\bullet p}$  is the Kummer radical of the first layers of  $\Z_p$-extensions of a particular type. One can show from \cite [Theorem 2.4]{FN91} that they are the  $\Z_p$-extensions  $K_\infty = \cup_{n \geq 0} K_n$ of  $F$, such that the rings of $p$-integers of all the  $K_n$'s have normal bases which are coherent for the trace. \\ 
(v) For  $i \neq 0$, the  $G_F$-modules  $\div({\mathfrak L}(i)^\Gamma)[p]$  are the generalized Tate kernels studied in 
\cite { 00, AM04, Hu05, Va08} in  connection with problems of capitulation. 
When  $F \supset \mu_{p^{e+1}}$, where  $p^e$  is the exponent of  $X^\circ$, then all the modules  
$\div({\mathfrak L}(i)^\Gamma)[p]$ are equal \cite {Gre78, AM04} (see also Corollary 3.2 below). 
For a general comparison between these kernels when  $i$  varies see \cite [Theorem 2.7]{Va08}.

Under Leopoldt's conjecture, the above Corollary 2.5  provides a good approximation of  ${\cal A}_F$  by  
${\tilde U}_F F^{\bullet p}/F^{\bullet p}$,  whose  
dimension over  $\F_p$ is  $r_2 - h$, where  $h:=dim {X^\circ}^\Gamma[p] = dim {X^\circ}(-1)^\Gamma[p]= dim (X^\circ(-1)^\Gamma/p).$ 
The difference  $h$  is of an asymptotic nature and bounded in the cyclotomic tower. 
Interpreting  $X^\circ$  as a capitulation kernel \cite {Iw73, Gre78}, this parameter is theoretically and effectively accessible. 

Nevertheless, the result is not entirely satisfactory as the $p$-units of  ${\tilde U}_F/p$  are not immediately   accessible. 
Kuz'min's exact sequence (Proposition 1.5)  suggests replacing  ${\tilde U}_F$  by  ${\hat U}_F$,  which is easily accessible by Lemma 1.4. 

\subsection{A lower bound in terms of local universal norms} 
We want to "approximate"   ${\cal A}_F$  by the intersection  ${\cal A}_F \cap {\hat U}_F/p$. 
In order to compute the deviation, let us come back to the exact sequence of Theorem 2.1,  where the map  
$\sigma_i : \div({\mathfrak L}(i)^\Gamma)[p] \twoheadrightarrow X^\circ(i)^\Gamma/p$  
is given by the snake lemma and hence depends on the twist  $i$. 
To compare the images of the  $\sigma_i$'s, we must put them in a space which does not (at least for the action of Galois groups over  $F$)  depend on  $i$. From the exact sequence 
$$0 \to X^\circ(i)^\Gamma  \to  {\mathfrak {\tilde N}}(i)^\Gamma \to {\mathfrak L}(i)^\Gamma $$  
(see Proposition 1.7), we derive 
$$0 \to (X^\circ)^\Gamma[p](i)  \to  {\mathfrak {\tilde N}}^\Gamma[p](i) \to {\mathfrak L}^\Gamma[p](i).$$  
Let  $W$  be the cokernel of the map on the right, 
so that  $W$  does not depend on  $i$  (for the Galois action over  $F$)  and we have a commutative diagram:

$$\xymatrix{
0 \ar[r] & (X^\circ)^\Gamma[p](i) \ar[r]  & {\mathfrak {\tilde N}}^\Gamma[p](i) \ar[r] & {\mathfrak L}^\Gamma[p](i) \ar[r] &  W  \ar[r]  &  0 \\
0 \ar[r] & (X^\circ)^\Gamma[p](i) \ar[r] \ar@{=}[u] &  {\mathfrak {\tilde N}}^\Gamma[p](i) \ar[r] \ar@{=}[u] & \div({\mathfrak L}(i)^\Gamma)[p]  \ar[r]^{\sigma_i} \ar@{^{(}->}[u] &   X^\circ(i)^\Gamma/p  \ar[r] \ar[u]_{\tau_i} & 0 \\
}$$
where the right vertical map  $\tau_i$  is defined tautologically and is injective. 
Finally, let  $T_i := \tau_i (X^\circ(i)^\Gamma/p) (-i)$, then with these notations we have 
\begin{prop}
Suppose that Leopoldt's and Gross' conjecture are valid for $F$.  
Then we have the following exact sequence: 
$$0 \to {\cal A}_F \cap {\hat U}_F/p  \to {\cal A}_F \to 
T_{-1}/ T_0 \cap T_{-1} \to 0. $$
\end{prop}

\begin{proof}
Write  $D_i:=\div {\mathfrak L}(i)^\Gamma[p]$  for short. 
Then  $D_i(-i)$,  for  $i=0, -1$,  can be identified with the Kummer radicals in  $F^\bullet/F^{\bullet p}$  of the Kummer extensions  $F(\sqrt[p]{{\hat U}_F})/F$  and  $F(\sqrt[p]{{\cal A}_F})/F$  respectively (obvious notations).  
According to Theorem 2.1,  $\tau_{i}(X^\circ(i)^\Gamma/p)$  for $i=0, -1$  can be identified with the Kummer radicals  in  $F(\sqrt[p]{{\tilde U}_F})$  of  $F(\sqrt[p]{{\hat U}_F})$  and $F(\sqrt[p]{{\cal A}_F})$  respectively (see the diagram). Hence the statement of the proposition by elementary Kummer theory. 
\vspace{2mm}
$$\xymatrix{
F(\sqrt[p]{{\cal A}_F}) \ar@{{}-{}}[r] & F(\sqrt[p]{{\cal A}_F {\hat U}_F}) \\
F(\sqrt[p]{{\hat U}_F}) \cap F(\sqrt[p]{{\cal A}_F}) \ar@{{}-{}}[r] \ar@{{}-{}}[u] & F(\sqrt[p]{{\hat U}_F}) \ar@{{}-{}}[u]  \\
F(\sqrt[p]{{\tilde U}_F})  \ar@{{}-{}}[u] \ar@{{}-{}}[u]&  \\
F  \ar@{{}-{}}[u] &   \\
}$$   
\end{proof}

{\bf  Remark}. 
An analogous result holds when replacing the pair  (${\cal A}_F,  {\hat U}_F/p$)  by any pair taken from  \{${\cal A}_F,  {\hat U}_F/p, {\cal T}_F$\}  or by any pair  ($D_i, D_j$), $i \neq j$. 
This should be compared with \cite [Theorem 2.7]{Va08} which states (in our notations) that  
$D_i /D_i \cap D_j \simeq p^t \Delta_{i,j}$, where  $t$ is the maximal integer such that  
$i \equiv j \mod [F(\mu_{p^t}) : F]$  and  $\Delta_{i,j} \subset  H^2(G_S(F), \Zp(j))$  
is the image, by corestriction, of  
$H^2_{Iw}(G_S(F_\infty), \Zp(i))^\Gamma(j-i) \subset H^2_{Iw}(G_S(F_\infty), \Zp(j))$.  
Here  $H^2_{Iw}(G_S(F_\infty), \bullet)$  denotes, as usual,  $\varprojlim H^2(G_S(F_n), \bullet)$  with respect to corestriction.

We already noticed (Remark (v) following Corollary 2.5) that  ${\cal A}_F$  and  ${\hat U}_F/p$  coincide when  $F$ contains enough  $p$-primary roots of unity (see also Corollary 3.2 below). Proposition 2.7  shows that this is not the case in general, but their deviation, which is of an asymptotical nature, goes to zero when we go up the cyclotomic tower.

\section{Upper bounds for the Kummer radical  ${\cal A}_F$} 
We are going to give two  "upper bounds"  for  ${\cal A}_F$. 
The first will be a "norm" radical which is accessible in the sense of the introduction. 
The second, via a local-global approach,  will be a radical characterized by everywhere local embeddability in a   $\Z_p$-extension. 

\subsection{Bounding from above by a norm radical} 
We keep the notations of the preceding sections. 

\begin{thm}
Suppose that  $F$  contains  $\mu_p$  and the layers  $F_n$'s verify Gross' conjecture  (i.e.  $\Delta_\infty =0$).  
Take  $m$  large enough for   $\Gamma_m$  to act  trivially on  $X^\circ$  and  put  $n=m+1$. 
Then, for every  $i \in \Z$, we have an exact sequence: 
$$0 \to \div ({\hat {\mathfrak N}}(i)^\Gamma) [p] \to ({\hat U}'_n/p)^{G_n}(i) \to X^\circ(i)_\Gamma [p] \to 0 .$$ 
In particular, if  $F$  also satisfies Leopoldt's conjecture, we have: 
 $$0 \to {\cal A}'_F \to ({\hat U}'_n/p)^{G_n}  \to (X^\circ(-1))_\Gamma [p](1) \to 0$$ 
where  $G_n =\gal(F_n/F)$. 
\end{thm} 

\begin{proof} 
Let us start from the exact sequence which appeared at the beginning of the proof of Theorem 2.1 
(notice that  by the hypothesis $\Delta_\infty =0$, we have  ${\mathfrak L} = {\hat {\mathfrak N}}$):  
\begin{equation}
\xymatrix{
0 \ar[r] & X^\circ(i)^\Gamma   \ar[r] &  {\mathfrak {\tilde N}}(i)^\Gamma   \ar[rr] \ar@{^{}->>}[rd] & & {\mathfrak {\hat N}}(i)^\Gamma  \ar[r] & 
X^\circ(i)_\Gamma  \ar[r] & 0 \\
  &   &   &  \div ({\mathfrak {\hat N}}(i)^\Gamma) \ar@{^{(}->}[ru] &  & &\\
}
\end{equation}
from which we derive by the snake lemma applied to the  $p$-th power map: 
$$0 \to \div ({\hat {\mathfrak N}}(i)^\Gamma) [p] \to {\hat {\mathfrak N}}(i)^\Gamma [p] \to X^\circ(i)_\Gamma [p] \to 0 .$$ 
It remains to give an adequate expression of  ${\hat {\mathfrak N}}(i)^\Gamma [p]$. 
But  
${\hat {\mathfrak N}}(i)^\Gamma [p] = {\hat {\mathfrak N}}^\Gamma [p](i)$  since  $F$  contains  $\mu_p$. 
Now, by Lemma 2.4, we have an exact sequence 
$$ 0 \to {\hat U}'_n\otimes \Q_p/\Z_p  \to {\mathfrak {\hat N}}^{\Gamma_n}  \to (X^\circ)_{\Gamma_n} \to 0,$$
for all $n$. Hence also 
$$ 0 \to {\hat U}'_n/p  \to {\mathfrak {\hat N}}^{\Gamma_n}[p]  \to (X^\circ)_{\Gamma_n}[p] \to 0.$$
Take the fixed points by  $G_n :=\gal (F_n/F)$ and compare with the level of  $F$.  
By Lemma 1.4, the natural map  $ {\hat U}'_F/p \to ({\hat U}'_n/p)^{G_n}$  is injective, hence we have a commutative diagram: 
$$\xymatrix{
0 \ar[r] & ({\hat U}'_n/p)^{G_n}  \ar[r]  &  {\mathfrak {\hat N}}^{\Gamma}[p] \ar[r] &
(X^\circ)_{\Gamma_n}[p]^{G_n}  &  \\
0 \ar[r] & ({\hat U}'_F/p) \ar[r] \ar@{^{(}->}[u] &  {\mathfrak {\hat N}}^{\Gamma}[p]  \ar[r] \ar@{=}[u] & 
(X^\circ)_{\Gamma}[p] \ar[r]  \ar[u] & 0. \\
}
$$
Let  $\omega_n :=\gamma^{p^n}-1$. The vertical map on the right (which corresponds to restriction) is multiplication by  $ \omega_n/\omega_0$. 
Choose  $m$  such that  $\Gamma_m$  acts trivially on  $X^\circ$,  and take  $n=m+1$, so that  
$\frac{\omega_n}{\omega_m}$  annihilates $(X^\circ)_{\Gamma_m}[p] = X^\circ[p]$. 
Then  $\frac{\omega_n}{\omega_0} = \frac{\omega_n}{\omega_m}$$\frac{\omega_m}{\omega_0}$  
is the zero map on  $(X^\circ)_{\Gamma}[p]$ and  we also get an isomorphism 
$({\hat U}'_n/p)^{G_n} \stackrel{\sim} \to {\mathfrak {\hat N}}^{\Gamma}[p]$. 
\end{proof}

The above Theorem gives a satisfactory description of the radical  ${\cal A}'_F$  in the terms set out in the introduction. The module  $({\hat U}'_n/p)^{G_n}$  is effectively accessible and the asymptotical deviation  $(X^\circ(-1))_{\Gamma}[p] (1)$  is bounded in the cyclotomic tower. 

For completeness, let us now reprove along the same lines a slightly improved version of a result of Greenberg, generalized by Vauclair: 

\begin{cor} (\cite [page 1242]{Gre78}, \cite [Theorem 2.2 ]{Va06}) 
Let   $p^e$  be the exponent of  $X^\circ$.  
Assume that  $\mu(F) = \mu_{p^a}$  with  $a \geq e$, 
and that the conditions of Theorem 3.1 hold. 
Then ${\cal A}_F = {\hat U}_F/p$  if  $a > e$.  
If  $a=e$  and  $\Gamma$  acts trivially on $X^\circ$,  then  ${\cal A}_F \neq {\hat U}_F/p$ 
\end{cor}

\begin{proof}
According to the proof of Theorem 3.1, we have two exact sequences 

$$\begin{array}{ccccccccccc} 
0  & \to & {\cal A}'_F & \to & {\mathfrak L}(-1)^\Gamma [p](1) = {\mathfrak L}^{\Gamma}[p]  & \to & (X^\circ(-1))_{\Gamma}[p](1) & \to & 0 \\
    &      & 		&  	 &  \parallel &  &   &  &  &  & \\
0  & \to & {\hat U}'_F/p  & \to & {\mathfrak {\hat N}}^{\Gamma}[p] & \to & (X^\circ)_{\Gamma}[p] & \to & 0 
\end{array} $$ 
which we need to link by vertical maps becoming equalities. 
These exact sequences were obtained by Kummer duality from  
$$\begin{array}{ccccccccccc} 
0  & \to &  H(1)^\Gamma/p & \stackrel{\sigma_1} \to & (\fr_{\Lambda} \mathfrak{X}_\infty(1))_\Gamma/p & \to & 
\fr_{\Z_p}((\fr_{\Lambda} \mathfrak{X}_\infty (1))_\Gamma)/p & \to & 0 \\
    &      & 		&  	 &  \parallel &  &   &  &  &  & \\
0  & \to & H^\Gamma/p  & \stackrel{\sigma} \to & (\fr_{\Lambda} \mathfrak{X}_\infty)_\Gamma/p & \to & 
\fr_{\Z_p}((\fr_{\Lambda} \mathfrak{X}_\infty )_\Gamma)/p & \to & 0 
\end{array} $$ 
where the maps  $\sigma$  and  $\sigma_1$  originate from the snake lemma.   
More precisely the map  
$\sigma : H^\Gamma/p \to (\fr_{\Lambda} \mathfrak{X}_\infty)_\Gamma/p = (\fr_{\Lambda} \mathfrak{X}_\infty)/(\omega,p)$  
(where $\omega=\gamma -1$) is defined from the exact sequence  
$$ 0 \to \fr_{\Lambda} \mathfrak{X}_\infty   \to \Lambda^{r_2} \to H \to 0$$ 
in the following way: for $h \in H^\Gamma$,  let  $\lambda$  be any lift of  $h$  in $ \Lambda^{r_2}$. 
Then  $\omega(\lambda) \in \fr_{\Lambda} \mathfrak{X}_\infty$  and  $\sigma$ sends  $h$ mod $p$ to $\omega(\lambda)$ mod $(\omega, p)$. 
Likewise, since  $H^\Gamma = H(1)^\Gamma$  by hypothesis,  we start with the same $h$ that we lift to the same  $\lambda \in \Lambda^{r_2}$.  
Put  $\omega^{(1)}:= \kappa(\gamma) \gamma -1.$
Then  $\omega^{(1)}(\lambda) \in \fr_{\Lambda} \mathfrak{X}_\infty$  and  $\sigma_1$ sends  $h$ mod $p$ to 
$\omega^{(1)}(\lambda)$ mod $(\omega^{(1)}, p)$. 
Now,  $(\omega^{(1)}, p) = (\omega, p)$  and  $(\omega^{(1)} - \omega)(\lambda) \in p^{e+1}\Lambda^{r_2}$ by hypothesis. Therefore  
$(\omega^{(1)} - \omega)(\lambda) \in p \fr_{\Lambda} \mathfrak{X}_\infty$  since  $p^e$ annihilates  $H$. 
We then have the following commutative square 
$$\begin{array}{ccccccccccc} 
H(1)^\Gamma/p & \stackrel{\sigma_1} \to & (\fr_{\Lambda} \mathfrak{X}_\infty(1))_\Gamma/p \\
\parallel    &      &    \parallel   \\
H^\Gamma/p & \stackrel{\sigma} \to & (\fr_{\Lambda} \mathfrak{X}_\infty)_\Gamma/p 
\end{array} $$ 
which implies  the equality  $\fr_{\Z_p}((\fr_{\Lambda} \mathfrak{X}_\infty (1))_\Gamma)/p = \fr_{\Z_p}((\fr_{\Lambda} \mathfrak{X}_\infty)_\Gamma)/p$, 
i.e. the equality  
${\cal A}'_F = {\hat U}'_F/p$.  
Hence  ${\cal A}_F = {\hat U}_F/p$. 

Suppose now that  $a=e$  and  $\Gamma$  acts trivially on  $X^\circ$. 
Then  $H(1)^\Gamma = H^\Gamma = H$.   
Keeping the same notations, choose  $h \in H$  of maximal order  $p^e$.  
Then  
$(\omega^{(1)}, p) = (\omega, p)$  and  $(\omega^{(1)} - \omega)(\lambda) \in p^e\Lambda^{r_2}$
and as before 
$(\omega^{(1)} - \omega)(\lambda) \in \fr_{\Lambda} \mathfrak{X}_\infty$. 
Putting  $\kappa(\gamma)=1+u p^e$, with  $u \in \Zp^*$,  we have  
$x:=(\omega^{(1)} - \omega)(\lambda) = u p^e \gamma(\lambda) \in \fr_{\Lambda} \mathfrak{X}_\infty$. 
If  $x \in p \fr_{\Lambda} \mathfrak{X}_\infty$,  we would get  
$p^{e-1} \gamma(\lambda) \in \fr_{\Lambda} \mathfrak{X}_\infty$, 
contrary to the choice of  $h$. 
We have thus shown that  $\sigma_1(h) \neq  \sigma (h)$ 
\end{proof}

{\bf Remark} 
Of course, an analogous result holds when replacing the pair  (${\cal A}_F,  {\hat U}_F/p$)  by any pair taken from  \{${\cal A}_F,  {\hat U}_F/p, {\cal T}_F$\}  or by any pair  ($D_i, D_j$)  as in the remark following Proposition 2.7, satisfying Greenberg's conjecture alluded to in the proof of Proposition 2.3. For a general exponent  $p^k$, see \cite [Theorem 2.2]{Va06}.

\subsection{Bounding from above by the Bertrandias-Payan radical} 

In this subsection, we introduce a certain Kummer radical  ${\cal B}_F$ (see \cite {Ng86})   coming from global-local considerations concerning embeddability in cyclic  $p$-extensions of arbitrary degrees. The radical  ${\cal B}_F$  contains all the previous ones  
${\cal A}_F,  {\cal T}_F$  and  ${\hat U}_F/p$. 
The determination of the respective quotients sheds additional light on the interrelationship between the three radicals themselves.

\begin{defn} 
A  (necessarily cyclic)  $p$-extension  $K/F$  is called {\it infinitely embeddable} (resp. $\Z_p$-embeddable) if it can be embedded in cyclic  $p$-extensions of arbitrarily large degrees (resp. in a  $\Z_p$-extension) of  $F$.  

The compositum of all the infinitely embeddable  $p$-extensions of  $F$  is called the {\it field of Bertrandias-Payan}  $F^{BP}$ (in reference to \cite {BP72}). 
An infinitely embeddable extension is necessarily unramified outside  $p$-adic primes:  $F^{BP} \subseteq {\hat F}$.  
Hence the Galois group  $BP_F :=\gal (F^{BP}/F)$  is a quotient of  ${\mathfrak X}_F$. 
Moreover  $F^{BP}$  obviously contains the composite of all  $\Z_p$-extensions of  $F$ so that 
$$\fr_{\Z_p} BP_F = \fr_{\Z_p} {\mathfrak X}_F.$$ 
\end{defn}

The following criterion is a consequence of class field theory (see \cite {BP72}):  

\begin{prop} 
Assume that  $F$ contains  $\mu_p = <\zeta_p>$  and let  $K:=F (\sqrt[p]{a})$, $a \in F^\bullet$  be a cyclic extension of degree  $p$. 
The following conditions are equivalent: \\
(i) $F (\sqrt[p]{a})/F$  is infinitely embeddable; \\
(ii) $F_v(\sqrt[p]{a})/F_v$  is $\Z_p$-embeddable for all finite primes  $v$; \\
(iii) $a \in F^{\bullet p}_v  {\tilde F}^\bullet_v$  for all finite primes  $v$  (note that  ${\tilde F}^\bullet_v= U_v$ for  $v \nmid p$); \\ 
(iv) $(a,\zeta_p)_v=1$  for all finite primes  $v$. Here $(\cdot,\cdot)_v$ stands for the maximal degree Hilbert symbol attached to the local field  $F_v$; \\
(v) the symbol $\{a,\zeta_p \}$  belongs to the wild kernel  $WK_2(F)$, 
i.e. the intersection in  $K_2(F)$  of the kernels of all the Hilbert symbols.
\end{prop}

\begin{defn} 
The radical of Bertrandias-Payan  ${\cal B}_F$  is, by definition, the subgroup of  $F^\bullet/F^{\bullet p}$ 
consisting of classes  $a$ mod  $F^{\bullet p}$  such that  $a$ verifies the preceding equivalent conditions. 
Clearly  ${\cal B}_F$  contains  ${\cal A}_F,  {\hat U}_F/p$  and the Tate kernel  ${\cal T}_F$. 
Condition (i)  means that  ${\cal B}_F = \hom(BP_F, \mu_p)$  and the quotient  ${\cal B}_F/{\cal A}_F$  
measures the obstruction between "global $\Z_p$-embeddability"  and "everywhere local $\Z_p$-embeddability". 
Condition (iii)  (resp. (iv), resp. (v))  says that  ${\cal B}_F$  coincides with the radical denoted by  
$D_F^{(1)}$ (resp. $B/F^{\bullet p}$,  resp. ${\cal E}/F^{\bullet p}$) in 
\cite {Ko91} (resp. \cite {Gre78}, resp. \cite [section 3]{KC78}). 
\end{defn}

\begin{thm} 
Suppose that  $F$ contains  $\mu_p$. Then \\
(i) ${\cal B}_F/{\cal T}_F \simeq X_\infty(1)_{\Gamma}[p] (-1)$ \\
(ii) ${\cal B}_F/{\cal A}_F \simeq X_\infty(-1)_{\Gamma}[p] (1)$ \\
(iii) ${\cal B}_F/({\hat U}_F/p) \simeq (X_\infty)_{\Gamma}[p].$
\end{thm}

\begin{proof} 
Tate's theorem \cite {Ta76}  asserts the existence of an exact sequence  
$$ 0 \to {\cal T}_F(1) \to  F^{\bullet} /F^{\bullet p}(1) \stackrel{\alpha} \to K_2(F)[p] \to 0,$$ 
where the map $\alpha$  is defined by  $\alpha(a \otimes \zeta_p)=\{a,\zeta_p \}$. 
Put  $H^1(G_S(F), \mu_p) = \Delta_F /F^{\bullet p},$   
where  $\Delta_F :=\{a \in F^{\bullet} /  v(a) \equiv 0 \pmod{p}  \; \forall v \nmid p \}.$  
Then, if we restrict the map  $\alpha$  to  $\Delta_F /F^{\bullet p}(1)$, the above exact sequence becomes 
$$ 0 \to {\cal T}_F(1) \to  \Delta_F /F^{\bullet p}(1) \stackrel{\alpha} \to R_2(F)[p] \to 0,$$ 
where  $R_2(F)$  denotes the tame kernel of  $F$, i.e.  the intersection in  $K_2(F)$  of the kernels of all the tame symbols. 
By property (v) of Proposition 3.4,  ${\cal B}_F$  is the inverse image of  $WK_2(F)[p]$  under the map $\alpha$, hence an exact sequence 
$$ 0 \to {\cal T}_F(1) \to  {\cal B}_F(1) \stackrel{\alpha} \to WK_2(F)[p] \to 0.$$ 
But it is known that the $p$-primary part of the  wild kernel  $WK_2(F) \{p\}$  can be expressed as a codescent module: 
the exact sequence
$$0 \to WK_2(F) \{p\} \to H^2(G_S(F), \Z_p(2)) \to \oplus_{v \mid p} H^2(F_v, \Z_p(2))$$
together with Poitou-Tate duality imply that  
$$WK_2(F) \{p\}^* \simeq Ker_S^1(F, \Q_p/\Z_p(-1)) := \ker (H^1(G_S(F), \Q_p/\Z_p(-1))\to \oplus_{v \mid p}H^1(F_v, \Q_p/\Z_p(-1))).$$ 
The cohomological triviality of  $\Q_p/\Z_p(-1)$  (Tate's Lemma) then shows that  
$$Ker_S^1(F, \Q_p/\Z_p(-1))^* \simeq X_\infty(1)_\Gamma.$$
Thus  (i)  is proved. 

(ii) The exact sequence 
$$ 0 \to \tor_{\Z_p} BP_F  \to  BP_F  \to \fr_{\Z_p} BP_F = \fr_{\Z_p} {\mathfrak X}_F \to 0,$$ 
yields, by Kummer duality, an exact sequence 
$$ 0 \to  {\cal A}_F  \to  {\cal B}_F  \to \hom (\tor_{\Z_p} BP_F, \mu_p) \to 0.$$ 
It remains to determine $ \hom (\tor_{\Z_p} BP_F, \mu_p) \simeq (\tor_{\Z_p} BP_F/p)^* (1)\simeq (\tor_{\Z_p} BP_F)^*[p](1)$. 
We will follow the proof of Theorem 4.2 of  \cite {Ng86} (but without appealing to Leopoldt's conjecture). 
The equivalence between the first two parts of Proposition 3.4 shows that  
$$(BP_F)^* =\{\chi \in {\mathfrak X}_F^* /  \chi_v \in \div({\mathfrak X}_v^*), \forall v \mid pÊ \}.$$ 
Here  ${\mathfrak X}_v$  is the Galois group over  $F_v$  of the maximal abelian  pro-$p$-extension of  $F_v$  and  
$\chi_v$  is the character obtained by restricting $\chi$  to  ${\mathfrak X}_v$. Then 
$$\begin{array}{llll}
(\tor_{\Z_p} BP_F)^* & \simeq (BP_F)^*/\div({\mathfrak X}_F^*)  \\
& \simeq \ker ({\mathfrak X}_F^*/\div \to \oplus_{v \mid p} {\mathfrak X}_v^*/\div)  \\
& \simeq \ker (H^1(G_S(F), \Q_p/\Z_p)/\div \to \oplus_{v \mid p} H^1(F_v, \Q_p/\Z_p)/\div).   
\end{array}$$
By a result of Tate in $p$-adic cohomology \cite [section 2]{Ta76}, the boundary map induces an isomorphism 
$H^1(G_S(F), \Q_p/\Z_p))/\div \simeq \tor_{\Z_p} H^2(G_S(F), \Z_p)$ 
(and similarly for the local cohomology groups), hence 
$$(\tor_{\Z_p} BP_F)^*[p] \simeq \ker (H^2(G_S(F), \Z_p) \to \oplus_{v \mid p} H^2(F_v, \Z_p))[p].$$
By Poitou-Tate's duality, 
$\ker (H^2(G_S(F), \Z_p) \to \oplus_{v \mid p} H^2(F_v, \Z_p))$ 
is dual to  $Ker_S^1(F, \Q_p/\Z_p(1))$ (the notation is analogous to that of (i)), and Tate's lemma once again shows that  
$Ker_S^1(F, \Q_p/\Z_p(1))^* \simeq X_\infty(-1)_\Gamma$. 
Thus  (ii)  is proved. 

(iii) 
Coming back to Sinnott's exact sequence, let us note that  $X_\infty$  is actually the Galois group over  $F_\infty$  of the maximal abelian pro-$p$-extension of  $F_\infty$  which is totally decomposed at every finite place of  $F_\infty$.  Hence, for any set of places  $T$  containing the  $p$-adic primes, we have the  $T$-analogue of Sinnott's exact sequence 
$${\bar U}_F^T \overset{} \longrightarrow {\tilde \oplus}_{v\in T}\bar{F}_v^{\bullet}/\tilde{F}_v^{\bullet} \overset{Artin} \longrightarrow  (X_\infty)_\Gamma \to A_F^T  \to 0,$$
where  $ {\bar U}_F^T$  denotes the  pro-$p$-completion of the group of  $T$-units of  $F$  and  
$A_F^T$  the  $p$-group of  $T$-classes of  $F$. 
Taking  $T$  to be the set of all finite primes of  $F$, we get an exact sequence  
$${\bar F^{\bullet}} \overset{g} \longrightarrow {\tilde \oplus}_{v\nmid \infty}\bar{F}_v^{\bullet}/\tilde{F}_v^{\bullet} \overset{Artin} \longrightarrow  (X_\infty)_\Gamma \to  0,$$
where  ${\bar F^{\bullet}}$  denotes the pro-$p$-completion of  $F^{\bullet}$, which injects into  
${\oplus}_{v\nmid \infty}\bar{F}_v^{\bullet}$  by Hasse's principle. 
The kernel of  $g$  consists of elements  $x \in {\bar F^{\bullet}}$  such that  $x \in {\tilde F}_v^{\bullet}$  for all finite primes  $v$.  
Since  ${\tilde F_v^{\bullet}} = U_v$   for any  $v \nmid p$, it follows at once that  $\ker g={\hat U}_F$. 
Applying the snake lemma for the multiplication by  $p$, we readily get a commutative diagram 
$$\xymatrix{
  &   &   & 0  \ar[d] &  \\
  &   &   & (X_\infty)_\Gamma[p]  \ar[d] &  \\
0 \ar[r] & \hat{U}'_F /p  \ar[r] &  F^{\bullet}/F^{\bullet p}  \ar[r] \ar@{.>}[dr] & \im g/p \ar[r] \ar[d] & 0 \\
  &   &   & {\tilde \oplus}_{v\nmid \infty} {F}_v^{\bullet}/\tilde{F}_v^{\bullet}  F_v^{\bullet p} &  \\
}$$ 
By property (iii)  of Proposition 3.4, the kernel of the dotted arrow is simply the radical of Bertrandias-Payan  ${\cal B}_F$  and a diagram chasing immediately leads to: 
$${\cal B}_F/({\hat U}_F/p) \simeq (X_\infty)_{\Gamma}[p].$$ 
Thus (iii)  is proved. 
\end{proof}

\textbf{Remarks}:
Idelic proofs of properties  (i)  and  (iii)  can be found in  \cite[pages 18 and 14]{Ko91}. 
Also, compare (ii) with the main result of  \cite {Se11}.

\section{Study of a particular case} 
Although essentially our approach has a  theoretical orientation, it lends itself to effective or algorithmic calculations, as will be shown in the examples below. 
\subsection{The case where  $X^\circ$  is of order  $p$} 
We continue to assume the two standard conjectures : Leopoldt's for our base field  $F$ and Gross' for all the layers  $F_n$. As noticed before (proof of Theorem 3.1), the short exact sequence at the infinite level 
$$0 \to X^\circ(i)  \to  {\mathfrak {\tilde N}}(i)  \to {\mathfrak {\hat N}}(i)  \to 0$$
yields 
\begin{equation}
\xymatrix{
0 \ar[r] & X^\circ(i)  \ar[r] &  {\mathfrak {\tilde N}}(i)^\Gamma   \ar[rr]^{f_i} \ar@{^{}->>}[rd] & & {\mathfrak {\hat N}}(i)^\Gamma  \ar[r] & 
X^\circ(i)  \ar[r] & 0 \\
  &   &   &  \div {\mathfrak {\hat N}}(i)^\Gamma \ar@{^{(}->}[ru] &  & &\\
}
\tag{2}
\end{equation}
which gives by applying the snake lemma to multiplication-by-$p$  and assuming that  $X^\circ$  is of order  $p$ 
$$0 \to X^\circ(i) \to  
{\mathfrak {\tilde N}}(i)^\Gamma[p]={\mathfrak {\tilde N}}^\Gamma[p](i) \to 
\div ({\mathfrak {\hat N}}(i)^\Gamma)[p] \to {X^\circ(i)} \to 0.$$
We are interested in  $\div ({\mathfrak {\hat N}}(i)^\Gamma)[p]$ for which the above exact sequence only provides a hyperplane. 
But in the exact sequence (2), the surjectivity of  
${\mathfrak {\tilde N}}(i)^\Gamma \to \div{\mathfrak {\hat N}}(i)^\Gamma $ 
shows that each  $\beta_i  \in \div ({\mathfrak {\hat N}}(i)^\Gamma)[p]$  has a pre-image  
${\tilde \beta}_i  \in {\mathfrak {\tilde N}}(i)^\Gamma[p^2]$.  In fact, we have a commutative diagram 

\begin{equation}
\xymatrix{
0 \ar[r] & X^\circ(i) \ar[r]  \ar[d]^= &  {\mathfrak {\tilde N}}(i)^\Gamma[p] \ar[r] \ar@{^{(}->}[d] & 
\div ({\mathfrak {\hat N}}(i)^\Gamma)[p] \ar[r]  \ar[d]^= & X^\circ(i) \ar[r]  & 0 \\ 
0 \ar[r] & X^\circ(i) \ar[r] \ar[d]^= & f_i^{-1}(\div ({\mathfrak{\hat N}}(i)^\Gamma)[p])  \ar[r]^{f_i} \ar@{^{(}->}[d] & 
\div ({\mathfrak {\hat N}}(i)^\Gamma)[p] \ar[r]  \ar@{^{(}->}[d] &  0 & \\ 
0 \ar[r] & X^\circ(i) \ar[r]  &  {\mathfrak {\tilde N}}(i)^\Gamma[p^2] \ar[r]^{g_i} & 
\div ({\mathfrak {\hat N}}(i)^\Gamma)[p^2] \ar[r] & X^\circ(i) \ar[r]  & 0 \\ 
}
\end{equation}
and we are going to devise an algorithm determining  ${\cal D}_i: = f_i^{-1}(\div ({\mathfrak{\hat N}}(i)^\Gamma)[p])$. 

To this end, we notice first that 
$$ {\mathfrak {\tilde N}}(i)^\Gamma[p^2] =  ({\mathfrak {\tilde N}}(i)[p^2]^{\Gamma_1})^{G_1} 
= ({\mathfrak {\tilde N}}^{\Gamma_1}[p^2](i))^{G_1} \simeq   ({\tilde U_1}'/p^2)(i)^{G_1}   \qquad \text{(Lemma 2.2)} $$ 
where  $\Gamma_1 := \gal (F_{\infty}/F_1)$  and  $G_1:=\gal (F_1/F)$. 
Since  ${\tilde U_1}'/p^2 \simeq ({\bar U'}_{\infty})_{\Gamma_1}/p^2 $  is free as a module over  $\Z/p^2[G_1]$  (Proposition 1.1), the fixed points  $({\tilde U_1}'/p^2)(i)^{G_1}$  can be identified with   
$\nu^{(i)}({\tilde U_1}'/p^2)$,  where  $\nu$  is the norm (trace in additive notation) relative to  $G_1$  and  
$\nu^{(i)}$  is its  $i$-twist. 
The above diagram (3)  then shows that  ${\cal D}_i$  consists of all the  $\nu^{(i)}({\bar x})$,  ${\bar x} \in {\tilde U_1}'/p^2$  such that  $p (g_i(\nu^{(i)}({\bar x})) = 0$, or, according to the third line of (3), such that  $p (\nu^{(i)}({\bar x})) \in X^\circ(i)$. 
Thus, we are left to express  $X^\circ(i)$  in this context. 
We suppose that  ${\hat U}'_F$  is known (by Lemma 1.4) hence also  ${\tilde U}'_F$  by Proposition 1.5.  
In fact, since they are both free $\Z_p$-modules with  $[{\hat U}'_F  :{\tilde U}'_F]=p$, there exists a  $\Z_p$-basis  
$(e_1, e_2, \cdots )$  of  ${\hat U}'_F$  such that  $(pe_1, e_2, \cdots )$  is a  $\Z_p$-basis of  ${\tilde U}'_F$.  
The first line of  (3) shows immediately that  $X^\circ(i)$  is the cyclic group generated by the class 
(in  ${\tilde U}'_F/p$)  of  $pe_1$, which can be written as  
$pe_1 \otimes 1/p \in  {\mathfrak {\tilde N}}(i)^\Gamma[p]$. 

Now, writing  ${\bar x} \in  {\tilde U_1}'/p^2$  as  
$x \otimes 1/p^2 \in {\mathfrak {\tilde N}}(i)^\Gamma[p^2]$,  with  $x \in {\tilde U_1}'$, we conclude our algorithm: 

\begin{prop}
The pre-image ${\cal D}_i= f_i^{-1}(\div ({\mathfrak{\hat N}}(i)^\Gamma)[p])$  consists of the elements 
$\nu^{(i)}(x) \otimes 1/p^2 \in {\mathfrak {\tilde N}}(i)^\Gamma[p^2]$, $x \in {\tilde U_1}'$  such that  $\nu^{(i)}(x) \otimes 1/p $  is colinear to  
$pe_1 \otimes 1/p$. 
\end{prop}

\textbf{Remarks}:
(i) The colinearity factor above is in  $\Z/p$. 
It is not zero precisely when  $\nu^{(i)}(x) \otimes 1/p^2$  is of order  $p^2$. \\
(ii) For practical purposes, it is important to stress that the maps  $f_i$  and  $g_i$  in diagram (3) are induced by the natural map    
${\mathfrak {\tilde N}} \rightarrow  {\mathfrak{\hat N}}$. \\
(iii) Our algorithm can be extended to the general case (i.e. when  $X^\circ$  is not necessarily of order  $p$),  but the calculations of course become heavier, because  $F_1$  must be replaced by some higher level  $F_n$,  and the preliminary determination of   $X^\circ$ (considered for example as a capitulation kernel; see the remark at the end of Section 2)  is not a trivial matter. \\
(iv) As in the already existing effective  (\cite {Gra85,He88}) or algorithmic (\cite {Th93})
calculations of  ${\cal A}_F$, we need information on the  $p$-units (via ${\hat U}'_F$)  and the  $p$-class groups  (via $X^\circ$). But note that the methods based on class field theory (op. cit.)  need further to appeal to explicit reciprocity laws (via Hilbert symbols).

\subsection{Biquadratic fields for  $p=3$} 
The biquadratic fields  $F:=\Q(\mu_3, \sqrt{d})$  for  $p=3$ with  $d \in \Z$  squarefree  and  $3 \nmid d$ were first studied by Kramer and Candiotti \cite {KC78} who determined   ${\cal A}_F$  for  $\mid d \mid <200$  and showed that  ${\cal A}_F = {\cal T}_F$,   except for the five critical values  $d=-107, 67, 103, 106$  and  $139$  which they did not treat.  
Immediately after, Greenberg \cite {Gre78} showed that  ${\cal A}_F \neq {\cal T}_F$  for $d=67$. 
The corresponding  ${\cal A}_F$  for  these critical cases were computed by H\'emard  
\cite {He88} using an idelic calculation of the  $\Z_p$-torsion of the Bertrandias-Payan module  $BP_F$ (see Section 3.2). A similar approach allowed Thomas  \cite {Th93} to devise an algorithm computing  ${\cal A}_F$  for  a wide range of $d$. 
We do not intend to redo these calculations. 
It seems more interesting to use our approach to determine the Tate kernel  ${\cal T}_F$   because, curiously enough, we know of no such systematic computation in the literature. 
For  $p=3$, there are only three relevant kernels  ${\hat U}_F/3$,   ${\cal A}_F$  and  ${\cal T}_F$  corresponding to the twists  $0, -1$ and  $1$  respectively.  
For the biquadratic fields  $F$, we aim to derive the Tate kernel from the knowledge of the two other kernels. 

Let us consider the five critical cases  $F:=\Q(\mu_3, \sqrt{d})$,  $d = -107, 67, 103, 106$  and  $139$  with  $p=3$. 
Using PARI package, we can see that  $X^\circ(F^+)$  is  of order  $3$,  hence the algorithm of Section 4.1  applies. 
It is easy to check that $F( \sqrt[3]{3})$  is the first layer of the anticyclotomic $\Z_3$-extension of  $F$ and 
$3$  is a global universal norm. The three distinct kernels  ${\hat U}'_F/3$,   ${\cal A}'_F$  and  ${\cal T}'_F$ have dimension 2 and contain  ${\bar 3}:=3 \;  \text {mod} \; F^{\bullet 3}$. 
Let us treat in detail a specific example, say  $d=67$. 
The fundamental unit of the quadratic subfield  $F^+ =\Q(\sqrt{67})$  is  $\epsilon = 48842 + 5967 \sqrt{67}$  
and  Lemma 1.4  shows that   ${\bar \epsilon} \in  {\hat U}'_F/3$  (but  $\epsilon \not \in {\hat U}'_F$). 
Then  ${\hat U}'_F/3 = <{\bar 3}, {\bar \epsilon}>$  whereas  
${\cal A}'_F  = <{\bar 3}, {\bar \epsilon^2}  {\bar \eta}>$ 
(\cite [page 371]{He88})  with  $\eta = 8 + \sqrt{67}$  being a generator of a  $3$-adic prime in  $F^+.$
These radicals are hyperplanes of the  $\Z/3$-vector space   ${\bar U}'_F/3 =<{\bar 3}, {\bar \epsilon},  {\bar \eta}>$,  but also of the $\Z/3$-vector space   $({\hat U}'_1/3)^{G_1}$  according to Theorem 3.1. 
Their explicit form  shows that a basis of  
$({\hat U}'_1/3)^{G_1}$  will be  $\{{\bar 3}, {\bar \epsilon}, {\bar \epsilon}^2 {\bar \eta} \}$, 
or better $\{{\bar 3}, {\bar \epsilon}, {\bar \eta}\}$. 
In particular,  $({\hat U}'_1/3)^{G_1}$  coincides with the image of the natural injection  
${\bar U}'_F/3  \hookrightarrow  {\bar U}'_1/3$.  
Obviously the third kernel  ${\cal T}'_F$  will be of the form  
$\{{\bar 3}, {\bar \eta} \}$  or  $\{{\bar 3}, {\bar \eta} {\bar \epsilon}\}$. 
But the choice between these two forms will unexpectedly be non trivial.

Since the rest of the calculation will essentially be elementary linear algebra, it will be more 
convenient to use additive notation for the three kernels 
$D_i:=\div({\mathfrak {\hat N}}(i)^\Gamma)[3]$. 
They could be put on the same level and computed using the algorithm of Section 4.1, but here we want to deduce the last one from the two others. 
Let us first fix some general notations: 
$D_i=<3 \otimes \frac{1}{3}, \beta_i >$,  
$\beta_0 = \epsilon \otimes \frac{1}{3}$, 
$\beta_{-1} = (\eta - \epsilon) \otimes \frac{1}{3}.$
According to diagram (3) of Section 4.1,  ${\cal D}_i = f_i^{-1}(D_i)$  is of type  $(3^2,3)$  and  
$f_i({\cal D}_i[3]) = f_i({\mathfrak {\tilde N}}(i)^\Gamma[3])$  is the cyclic group generated by  $3 \otimes \frac{1}{3}$.  We look for pre-images (necessarily of order $3^2$)  ${\tilde \beta}_i \in {\cal D}_i$  of  $\beta_i$  such that  
${\cal D}_i = <3 \otimes \frac{1}{3}, {\tilde \beta}_i >$.

Recall that Proposition 4.1  gives us the general form of the  ${\tilde \beta}_i$'s, but here we directly know  (by our chosen approach)  $\beta_0$  and  $\beta_{-1}$.  
Let us fix  $\epsilon' \in {\hat U}'_F$  such that  
$\epsilon' \otimes \frac{1}{3} = \epsilon \otimes \frac{1}{3}$  in  ${\mathfrak {\hat N}}$. 
We have just seen that  $f_0({\mathfrak {\tilde N}}^\Gamma[3]) = <3 \otimes \frac{1}{3}>$,  
hence  $\epsilon' \otimes \frac{1}{3} \not \in {\mathfrak {\tilde N}}^\Gamma[3]$  
or equivalently  $\epsilon' \not \in {\tilde U}'_F$.  Because  $3 \epsilon' \in {\tilde U}'_F$,  a pre-image  
${\tilde \beta}_0$  of  $\beta_0$  will be  $3 \epsilon' \otimes \frac{1}{3^2}$. 
Similarly, consider   $\eta \otimes \frac{1}{3} \in ({\hat U}'_1/3)^{G_1}$  and fix  
$\eta' \in {\hat U}'_1$  such that  $\eta' \otimes \frac{1}{3} = \eta \otimes \frac{1}{3}$. 
We have seen that  $({\tilde U}'_1/3)^{G_1} = \nu({\tilde U}'_1/3)$  and therefore  
$({\tilde U}'_1/3)^{G_1} \simeq {\tilde U}'_F/3$  
since the (arithmetic) norm of 
${\tilde U}'_1/3$  is  ${\tilde U}'_F/3$.  
The same argument as for  $\epsilon'$  allows us to show that  $\eta' \not \in {\tilde U}'_1$,  
and the natural map  
${\mathfrak {\tilde N}}[3^2] \to {\mathfrak {\hat N}}[3^2]$  (whose kernel is  $X^\circ$)  
sends  $3(\eta' - \epsilon') \otimes \frac{1}{3^2}$  to  
$(\eta - \epsilon) \otimes \frac{1}{3}$,  
but be aware that  $3(\eta' - \epsilon') \otimes \frac{1}{3^2}$  is not a priori in 
${\mathfrak {\tilde N}}(-1)^\Gamma[3^2]$. 
Actually, Proposition 4.1  shows that there exists an element  
$\nu^{(-1)}(x) \otimes \frac{1}{3^2} \in {\cal D}_{-1}$  such that  
$g_{-1}(\nu^{(-1)}(x) \otimes \frac{1}{3^2}) = (\eta - \epsilon) \otimes \frac{1}{3^2}$, 
hence  $\nu^{(-1)}(x) \otimes \frac{1}{3^2}$  and   $3(\eta' - \epsilon') \otimes \frac{1}{3^2}$  
differ by an element of order  $3$  in  ${\mathfrak {\tilde N}}[3^2]$.
Therefore, we can take  
${\tilde \beta}_{-1} = 3(\eta' - \epsilon') \otimes \frac{1}{3^2} + \delta_{-1}$  with  
$\delta_{-1} \in {\mathfrak {\tilde N}}[3]$. 
The additional condition in Proposition 4.1 then reads: 
$3(\eta' - \epsilon') \otimes \frac{1}{3} = \pm 3 \epsilon' \otimes \frac{1}{3}$. 
The sign  $-1$  on the right hand side would mean that  $3 \eta' \otimes \frac{1}{3} = 0$  in  
${\tilde U}'_1/3$, i.e.   $\eta' \in {\tilde U}'_1$  since  ${\tilde U}'_1$  is torsion free: a contradiction. 
Finally,  ${\tilde \beta}_{-1} = 3(\eta' - \epsilon') \otimes \frac{1}{3^2} + \delta_{-1}$  and  
$3 {\tilde \beta}_{-1} = 3 \epsilon' \otimes \frac{1}{3}$. 

We must now give a general expression for the elements of order  $3^2$  in  ${\cal D}_1$. 
If  ${\tilde \beta}_1$  is such an element, any other is obtained by adding  an element of order  $3$  to  $\pm {\tilde \beta}_1$. 
So we start by constructing a particular  ${\tilde \beta}_1$. 
Fix  $x_0, x_{-1} \in {\tilde U}'_1$  such that  
$\nu(x_0) \otimes \frac{1}{3} = {\tilde \beta}_0$  and  
$\nu^{(-1)}(x_{-1}) \otimes \frac{1}{3} = {\tilde \beta}_{-1}$ (see Proposition 4.1), and put  
$x_1=x_0 + \lambda x_{-1}$, $\lambda = \pm 1$. 
An elementary calculation on Tate twists shows that  
$\nu^{(1)}(x_i)$  differs from  $\nu^{(i)}(x_i)$  by an element of  $({\tilde U}'_1)^3$,  hence  
${\tilde \beta}_1 := \nu^{(1)}(x_1)\otimes \frac{1}{3} 
= {\tilde \beta}_0 + \lambda {\tilde \beta}_{-1}+ \delta_1$, with  
$\delta_1 \in {\mathfrak {\tilde N}}[3]$. 
The additional condition in  Proposition 4.1 reads: 
$3 {\tilde \beta}_1 = 3 {\tilde \beta}_0 + \lambda 3 {\tilde \beta}_{-1} 
= \pm (3 \epsilon' \otimes \frac{1}{3})$.  
Since  $3 {\tilde \beta}_0 = 3 \epsilon' \otimes \frac{1}{3} = 3 {\tilde \beta}_{-1}$  (see the previous calculations), we get  $1+\lambda = \pm 1$. The only possibility is  $\lambda = 1$. 
Hence  
${\tilde \beta}_1 = {\tilde \beta}_0 + {\tilde \beta}_{-1} + \delta_1$  
and  $f_1({\tilde \beta}_1) = \eta \otimes \frac{1}{3} + f_1(\delta_1)$,  
so that any element of  $D_1$  will be of the form  $\pm \eta \otimes \frac{1}{3} + \delta'_1$,  
where  $\delta'_1$  is in the image of the natural map  
${\mathfrak {\tilde N}}[3] \to {\mathfrak {\hat N}}[3]$.  
We can now choose between the two possibilities 
$\eta \otimes \frac{1}{3}$  or  $(\eta + \epsilon) \otimes \frac{1}{3}$.  
In the second case, we would have  $\epsilon \otimes \frac{1}{3}$  or  
$(\epsilon - \eta) \otimes \frac{1}{3} = \delta'_1$, whence 
${\tilde \beta}_0$  or  ${\tilde \beta}_{-1}$  would be of order  $3$ (because the kernel  $X^\circ$  is of order  $3$): a contradiction. 
In conclusion, the Tate kernel in our example is  ${\cal T}'_F = <{\bar 3}, {\bar \eta}>$:  
the symbol  $\{\zeta_3,  \eta\}$  is trivial in  $K_2(F)$   whereas  
$\{\zeta_3,  \epsilon \eta \}= \{\zeta_3,  \epsilon \}$  is a non-trivial element of the wild kernel   $WK_2(F).$

\vspace{5mm}
\textbf{Remark}: 
The above result could of course be reached by describing the whole tame kernel  $K_2 o_F$  by generators and relations. Algorithms for such a calculation exist in the literature \cite {BG04}. 
However, Karim Belabas kindly pointed out to us that in the case of the above example, the generators are  $S$-units, where  $S$  contains all the primes of norm less than  $5096521$ so by brute force calculation it would take years to find the relations between them.

\vspace{2cm}
\begin{tabular}{ll}
 \textsc{A. Movahhedi} & \hspace{2cm} \textsc{T. Nguyen Quang Do} \\
Universit\'e de Limoges                        & \hspace{2cm} Universit\'e de Franche-Comt\'e \\
XLIM                                            & \hspace{2cm} Laboratoire de Mathématiques \\
UMR 7252 CNRS                                  & \hspace{2cm} CNRS UMR 6623 \\
123 av. Albert Thomas                           & \hspace{2cm} 16 route de Gray \\
87060 LIMOGES CEDEX                            & \hspace{2cm} 25030 BESANÇON CEDEX \\
FRANCE                                         & \hspace{2cm} FRANCE \\
\end{tabular}
\end{document}